\documentclass[11pt,a4paper]{article}
\usepackage{amsmath,amsthm,amssymb,amsfonts,color,stmaryrd}
\usepackage{hyperref}
\usepackage{graphicx}
\newcommand{\C}{{\mathbb{C}}}          
\newcommand{\N}{{\mathbb{N}}}          
\newcommand{\R}{{\mathbb{R}}}          

\newcommand{\Gdois}{{\mathrm{G}_2}}
\newcommand{\Spin}{{\mathrm{Spin}}}
\newcommand{\SO}{{\mathrm{SO}}}
\newcommand{\SU}{{\mathrm{SU}}}
\newcommand{\Uni}{{\mathrm{U}}}

\newcommand{\lrr}{\longrightarrow}

\newcommand{\calri}{{{\cal R}^U}}             %
\newcommand{\cals}{{\cal S}}             %
\newcommand{\calt}{{\cal T}}             %

\newcommand{\na}{{\nabla}}

\newcommand{\End}[1]{{\mathrm{End}}\,{#1}}

\newcommand{\dx}{{\mathrm{d}}}

\newcommand{\inv}[1]{{#1}^{-1}}

\newcommand{\vol}{{\mathrm{vol}}}

\newcommand{\ric}{{\mathrm{Ric}\,}}
\newcommand{\scal}{{\mathrm{scal}\,}}
\newcommand{\gsudois}[2]{{g_{_{\SU(2)}}(#1,#2)}}

\newtheorem{teo}{Theorem}[section]
\newtheorem{lemma}{Lemma}[section]
\newtheorem{coro}{Corollary}[section]
\newtheorem{prop}{Proposition}[section]
\newtheorem{defi}{Definition}[section]

\newenvironment{Rema}[1][Remark.]{\begin{trivlist}
\item[\hskip \labelsep {\bfseries #1}]}{\end{trivlist}}

\def\cyclic{\mathop{\kern0.9ex{{+}
\kern-2.2ex\raise-.28ex\hbox{\Large\hbox{$\circlearrowright$}}}}\limits}

\title{Natural $\SU(2)$-structures on tangent sphere bundles}

\author{R. Albuquerque\\ rpa@uevora.pt}

\begin{document}

\maketitle


\begin{abstract}

We define and study natural $\SU(2)$-structures, in the sense of Conti-Salamon, on the total space $\cals$ of the tangent sphere bundle of any given oriented Riemannian 3-manifold $M$. We recur to a fundamental exterior differential system of Riemannian geometry. Essentially, two types of structures arise: the contact-hypo and the non-contact and, for each, we study the conditions for being hypo, nearly-hypo or double-hypo. We discover new double-hypo structures on $S^3\times S^2$, of which the well-known Sasaki-Einstein are a particular case. Hyperbolic geometry examples also appear. In the search of the associated metrics, we find a theorem, useful for explictly determining the metric, which applies to all $\SU(2)$-structures in ge\-ne\-ral. Within our application to tangent sphere bundles, we discover a whole new class of metrics specific to 3d-geometry. The evolution equations of Conti-Salamon are considered, leading us to a new integrable $\SU(3)$-structure on $\cals\times\R_+$ associated to any flat $M$.

\end{abstract}

\ 
\vspace*{3mm}
\vspace*{2mm}

\vspace*{4mm}

\markright{\sl\hfill  R. Albuquerque \hfill}

\vspace*{10mm}

\section{The fundamental exterior differential system}

\subsection{Introduction}
\label{sec:Introduction}

The notion of $\SU(2)$-structure was introduced by Conti and Salamon in \cite{ContiSalamon} and plays an important role in the theory of generalized Killing spinors. It consists of the geometrical data induced on any hypersurface of a real 6-dimensional manifold endowed with an integrable special-Hermitian or $\SU(3)$-structure. For the latter we refer for instance to \cite{Agri}.

$\SU(2)$-structures become an independent notion on real 5-manifolds $N$. They are given by three 2-forms $\omega_1,\omega_2,\omega_3$ and a contact 1-form $\theta$ satisfying certain relations between them. These forms induce a Riemannian metric on $N$ and a canonical $\SU(2)$-structure on $\ker\theta$. The present article discovers a very useful result concerning the deduction of such metric. Theorem  \ref{Teo_themetriconLperp} gives the following identity for the metric
 $g_{_{\SU(2)}}$ on $\ker\theta$, where $v$ is the volume form ($v=\frac{1}{2}\omega_i\wedge\omega_i$, $\forall i$), which indeed seems to be new:
 \begin{equation}\label{themetriconLperp_introduction}
  x\lrcorner\omega_1\wedge y\lrcorner\omega_2\wedge\omega_3=\gsudois{x}{y}\,v, \ \ \ \forall  x,y\in \ker\theta.
 \end{equation}

For hypersurfaces $N$, the induced $\SU(2)$-structure is hypo, i.e. satisfies the equations
\begin{equation}
 \dx\omega_1=0,\qquad\dx(\theta\wedge\omega_2)=0,\qquad\dx(\theta\wedge\omega_3)=0.
\end{equation}
Conti and Salamon prove the `embeding property', which is almost a reciprocal: an \textit{analytic} $\SU(2)$-structure satisfying the hypo system admits an embedding into an integrable special Hermitian manifold. This may eventually be compact, hence a Calabi-Yau 3-fold.

The hypo, the nearly-hypo, and other particular differential systems, imply interesting geometry on $N$. For instance, one easily meets with Sasaki-Einstein manifolds.

This article studies the question raised in \cite{Alb2015a} regarding a certain $\SU(2)$-structure defined on the total space $\cals$ of the tangent sphere bundle of a given oriented Riemannian 3-manifold $M$. This shall be referred as the \textit{main example}. We generalise the construction with what we call the natural structures, supported by the fundamental exterior differential system of Riemannian geometry introduced in \cite{Alb2011arxiv} and \cite{Alb2015a}. A classification of $\SU(2)$-structures according to first derivatives, which we then follow, was developed in \cite{BedulliVezzoni} and \cite{FIMU}, two references well acquainted with the foundational article of Conti and Salamon.

The exterior differential system discovered in \cite{Alb2011arxiv} depends only on the orientation and the metric on $M$. It consists, in general, of a natural contact 1-form $\theta$ and set of natural differential $n$-forms $\alpha_0,\ldots,\alpha_n$ existing always on the total space $\cals$ of the (unit) tangent sphere bundle $SM\lrr M$ of any given oriented Riemannian $n+1$-manifold $M$. Of course, $\cals$ inherits the induced metric from the well-known canonical or Sasaki metric on $TM$ (not to be confused with Sasakian or Sasaki-Einstein metrics below). The metric plays a central role in defining the $\alpha_0,\ldots,\alpha_n$. The compatible contact structure $\theta$ is due to Tashiro. Applications of the natu\-ral differential system are discussed in \cite{Alb2011arxiv}.

Here we shall consider just the case $n=2$, so that $\cals$ is a 5-dimensional manifold. The fundamental differential system brings up four pairwise-orthogonal 2-forms $\alpha_0,\alpha_1,\alpha_2,\dx\theta$, sa\-tis\-fying: 
\begin{equation}\label{introd_derivadasdastres2formas} 
\begin{split}
  \dx\alpha_1=2\,\theta\wedge\alpha_2-r\,\theta\wedge\alpha_0 ,  \hspace{12mm}\: \\
  *\theta=\alpha_0\wedge\alpha_2=-\frac{1}{2}\,\alpha_1\wedge\alpha_1=-\frac{1}{2}\,\dx\theta\wedge\dx\theta , \\
  \dx\alpha_0=\theta\wedge\alpha_1,  \:\qquad\quad
 \dx\alpha_2  = \calri\alpha_2  . \hspace{6mm} 
\end{split}
\end{equation}
The function $r=r(u)=\ric(u,u),\ u\in\cals$, and the 3-form $\calri\alpha_2$ are curvature dependent tensors. For constant sectional curvature $K$ we have $r=2K$ and $\calri\alpha_2=-K\,\theta\wedge\alpha_1$.

In dimension 3 we have the \textit{nice} coincidence that the $\alpha_i$ are 2-forms like $\dx\theta$, and then an $\SU(2)$-structure naturally takes place. The \textit{main example} is
\begin{equation}
 \omega_1=\dx\theta,\qquad\omega_2=\alpha_2-\alpha_0,\qquad\omega_3=\alpha_1
\end{equation}
but many other linear combinations give interesting structures as well. Two distinct types appear with different properties. The distinction seems to be chiefly between those for which $\dx\theta$ is in the linear span of the $\omega_1,\omega_2,\omega_3$, and those for which it is out. Our further results here concentrate more on the first type.

In \cite{FIMU}, we see that Fern\'andez, Ivanov, Mu\~noz and Ugarte also discovered $\SU(2)$-structures on $\cals$ for $M=S^3$, which is the Stiefel-manifold $V_{4,2}$. Our coordinate-free tools lead us to generalise mildly one of those results and also to rediscover the Sasaki-Einstein well-known metrics. More important, \cite{FIMU} introduces the nearly-hypo and double-hypo structures, which have very deep relations with nearly-K\"ahler manifolds and half-flat $\SU(3)$-structures. As it is well-known, the latter yield true $\Gdois$-manifolds. Since we have found below new families of double-hypo structures associated to hyperbolic base $M$, they should lead to interesting results inspired by \cite{FIMU}. New developments from our construction and technique shall be continued in the near future.

Finally, we recall the evolution equations, again due to Conti and Salamon, and solve them in one particular family of hypo manifolds. This leads us to a new integrable $\SU(3)$-structure, associated to any given flat base $M$, defined on the space $\cals\times\R_+$. Which is not a trivial analytic manifold.

\subsection{The differential system}
\label{sec:Thedifferentialsystem}

We briefly recall the theory from \cite{Alb2011arxiv}. Let $M$ denote any oriented $n+1$-dimensional Riemannian manifold. Then the space $\calt_M$, i.e. the total space of the vector bundle $\pi:TM\lrr M$, is well-known to be a smooth manifold of dimension $2n+2$. In natural coordinates, we may identify $V:=\ker\dx\pi$ with $\pi^\star TM$, the tangent to the fibres. Taking the Levi-Civita connection $\na:\Gamma(M;TM)\lrr\Gamma(M;T^*M\otimes TM)$, we get the canonical decomposition of $T{\calt_M}=H\oplus V\simeq\pi^*TM\oplus\pi^\star TM$. The connection dependent horizontal distribution $H$ identifies again with $\pi^*TM$ via $\dx\pi$. Hence there exists a vector bundle endomorphism $B:T\calt_M\lrr T\calt_M$ which sends horizontals to verticals and verticals to 0; it is called the \textit{mirror} map. Most important is that $B$ is parallel for the pull-back connection $\na^*$ by construction. We let $B^{\mathrm{t}}$ denote the adjoint of $B$.

There are two canonical vector fields on $\calt_M$. The first is the tautological vertical vector field $U$, defined by $U_u=u,\ \forall u\in \calt_M$; it is the independent mirror of the second, the geodesic spray, defined on the horizontal distribution and hence connection dependent. Any given frame in $H$, followed by its mirror in $V$, clearly determines a unique orientation on the manifold $\calt_M$.

We recall the map $J=B-B^{\mathrm{t}}$ gives the well-known canonical or Sasaki almost complex structure on $\calt_M$. 

Next we consider the well-known canonical or Sasaki metric on the $2n+2$-manifold $\calt_M$. The mirror map becomes an isometry. Any frame at point $u$ arising from an orthonormal frame in $H$ with the first vector equal to $B^tU/\|U\|$, together with the mirror frame in $V$, in fixed order `first $H$, then $V$', is said to be an \textit{adapted frame} of $\calt_M$.

We hence find that $\calt_M\backslash$(zero section) has structure group the Lie group $\SO(n)$, cf. \cite[Theorem 1.1]{Alb2015a}. A representation of $\SO(n)$, acting diagonally, occurs on the common orthogonal distribution to the geodesic spray and to $U$. On these two directions, the action is of course trivial.

Now we consider the constant radius $s$ tangent sphere bundle of $M$
\begin{equation}
 {S}_s{M}=\{u\in TM:\ \|u\|=s\}.
\end{equation}
We let $\cals=\cals_{s,M}$ denote the total space of $S_sM$. We have
$T{\cals}=U^\perp\subset T\calt_M$, because $\ker\na^*_\cdot U=H$ and $\na^*_vU=v,\ \forall v\in V$. In particular, this manifold is orientable. The Riemannian submanifold $\cals$ inherits the $\SO(n)$-structure, which however is never parallel because $U$ is not parallel.

From the above remarks we have that any orthonormal frame $u,e_1,\ldots,e_n$ on $M$ induces by horizontal and vertical lifts an adapted frame $e_0,e_1,\ldots,e_n,e_{n+1},\ldots,e_{2n}\in T_u\cals$ at point $u\in\cals$, where $e_0=\frac{1}{s}B^{\mathrm{t}}U_u\in H_u$.

We denote by $\theta$ the 1-form on $\cals$ defined as
\begin{equation}\label{thecontactform}
\theta=\langle U,B\,\cdot\,\rangle=s\,e^0.
\end{equation}
It is well-known that $\theta$ and $J$ define a metric contact structure on $\cals$. We also recall the result $\dx\theta=e^{(1+n)1}+\cdots+e^{(2n)n}$ (from our usual notation: $e^{ij}=e^i\wedge e^j$ and this has norm 1). This is reminiscent of the Liouville form on $T^*M$, and one sees the amazing fact that $\dx\theta$ no longer depends on $s$.

The $\SO(n)$-structure induces the following natural fundamental differential system discovered in \cite{Alb2011arxiv} of global $n$-forms $\alpha_0,\alpha_1,\ldots,\alpha_n$ on $\cals$.

We first write $\pi^\star\vol_{_M}$ for the vertical lift of the volume form of $M$ (always a $\pi^\star$ denotes a vertical lift). Then
\begin{equation}\label{alphazero}
 \alpha_n = \frac{1}{s}\,U\lrcorner({\pi}^\star\vol_{_M}) 
\end{equation}
and for each $0\leq i\leq n$ we define, $\forall v_1,\ldots,v_n\in T\cals$,
\begin{equation}\label{alpha_itravez}
 \alpha_i(v_1,\ldots,v_n) = \frac{1}{i!(n-i)!}\sum_{\sigma\in\mathrm{Sym}(n)}\mathrm{sg}(\sigma)\,  \alpha_n(Bv_{\sigma_1},\ldots,B v_{\sigma_{n-i}},v_{\sigma_{n-i+1}},\ldots,v_{\sigma_n}).
\end{equation}
For convenience one also writes $\alpha_{-1}=\alpha_{n+1}=0$; we use the notation
\begin{equation}
 R_{lkij} = \langle R^\na(e_i,e_j)e_k,e_l\rangle= 
 \langle \na_{e_i}\na_{e_j}e_k-\na_{e_j}\na_{e_i}e_k-\na_{[e_i,e_j]}e_k,e_l\rangle.
\end{equation}
\begin{teo}[1st-order structure equations, \cite{Alb2011arxiv}] \label{derivadasdasnforms}
We have
\begin{equation}\label{dalphai}
 \dx\alpha_i=\frac{1}{s^2}(i+1)\,\theta\wedge\alpha_{i+1}+\calri\alpha_i
\end{equation}
where
\begin{equation}\label{Ralphai}
 \calri\alpha_i = \sum_{0\leq j<q\leq n}\sum_{p=1}^nsR_{p0jq}\,e^{jq}\wedge
e_{p+n}\lrcorner\alpha_i.
\end{equation}
\end{teo}
Defining $r=\frac{1}{s^2}\pi^\star\ric(U,U)=\sum_{j=1}^nR_{j0j0}$, a smooth function on $\cals$ determined by the Ricci curvature of $M$, we find that $\calri\alpha_0=0$ and $\calri\alpha_{1}=-r\,\theta\wedge\alpha_0$. This is
\begin{equation}\label{dalphanen-1}
 \dx\alpha_0=\frac{1}{s^2}\,\theta\wedge\alpha_{1},\qquad\quad
 \dx\alpha_{1}=\frac{2}{s^2}\,\theta\wedge\alpha_{2}-sr\,\vol.
\end{equation}
Moreover, the differential forms \,$\theta$,\ $\alpha_n$ and $\alpha_{n-1}$ are always coclosed. In every degree we have
\begin{equation}
 \alpha_i\wedge\dx\theta=0.
\end{equation}
No further assumptions are required, besides orientation and a metric, in order to find the fundamental exterior differential system $\{\theta,\alpha_0,\ldots,\alpha_n\}\Omega^*_\cals$ associated to a given oriented Riemannian manifold.

\subsection{The 3d differential system}
\label{T3dc}

We now consider a 3-dimensional $M$ together with the total space $\cals$ of the tangent 2-sphere bundle of radius $s$ equipped with canonical metric and  orientation. We have the contact 1-form, $\theta=s\,e^0$,  clearly invariant for the action of $\SO(2)$ on $\R^{1+2+2}$, i.e. the trivial action on the 1-dimensional summand and the diagonal action on $\R^{2+2}$.

The global invariant 2-forms, independent of the choice of adapted frame, are
\begin{equation}\label{thefourinvariants}
 \alpha_0=e^{12},\qquad \alpha_1=e^{14}-e^{23},\qquad \alpha_2=e^{34},\qquad\dx\theta=e^{31}+e^{42}.
\end{equation}
We also have
\begin{equation} 
  \alpha_0\wedge\alpha_1=\alpha_2\wedge\alpha_1=\alpha_i\wedge\dx\theta=0,\:\ \forall i=0,1,2,
\end{equation}
and
\begin{equation}\label{basicstructurequations}
   \frac{1}{s}\,*\theta=\alpha_0\wedge\alpha_2=-\frac{1}{2}\,\alpha_1\wedge\alpha_1=-\frac{1}{2}\,\dx\theta\wedge\dx\theta.
\end{equation}
\begin{prop}[\cite{Alb2015a}]\label{prop_decompwedgetwo}
 The representation under $\SO(2)$ above, induced on the vector bundle $\Lambda^2T^*\cals$, corresponds with the decomposition
\begin{equation}\label{decompwedgetwo}
 \Lambda^2\R^5=4\R^1\oplus W_1\oplus W_2\oplus W_3
\end{equation}
where we have the four 1-dimensional invariants from \eqref{thefourinvariants} and three irreducible orthogonal subspaces $W_i$ defined by
\begin{equation}
 W_1=\llbracket e^{01},e^{02}\rrbracket, \qquad W_2=\llbracket e^{03},e^{04}\rrbracket,\qquad
 W_3=\llbracket \psi_1,\psi_2\rrbracket
\end{equation}
where
\begin{equation}\label{The_f_forms}
 \psi_1:=e^{14}+e^{23},\qquad \psi_2:=e^{31}-e^{42}.
\end{equation}
\end{prop}
Other $\Lambda^pT^*\cals$ are easily decomposed. Since the canonical map $\Lambda^1\R^5\otimes\Lambda^2\R^5\lrr\Lambda^3\R^5$ has a kernel of dimension 40, there are many equivalent representations in the space of 3-forms.

The scalar function $r=\frac{1}{s^2}\pi^\star\ric(U,U)=R_{1010}+R_{2020}$ (recall 0 stands for the point $u\in\cals$) may be written using scalar and sectional curvatures as $r=\frac{1}{2}\scal-K(\{e_1,e_2\})=\frac{1}{2}\scal-R_{1212}$. If $M$ is Einstein, this is $\ric=\lambda\langle\cdot,\cdot\rangle$ for some constant $\lambda$, then clearly $M$ has constant sectional curvature $\lambda/2$.

Now recall we have the Sasaki almost complex structure $J$ on $H_0\oplus V_0$, where $H_0=H\cap e_0^\perp=\llbracket e_1,e_2\rrbracket$ and $V_0=V\cap U^\perp=\llbracket e_3,e_4\rrbracket$ are sub-vector bundles of $T\cals$. We may further define $I_+$ and $I_-$, according to $\pm$, to be the unique map defined on any adapted frame as
\begin{equation}\label{segundaestruturacomplexa}
  e_0\mapsto 0,\qquad e_1\mapsto e_2 \mapsto-e_1,\qquad
   e_3\mapsto \pm e_4 \mapsto -e_3.
\end{equation}
$I_+,I_-$ are commuting endomorphisms of $T\cals$. On one hand,
$JI_+J^t=JI_+\inv{J}=I_+$. On the other, we have that $J$ and $I_-$ anti-commute, giving an $\mathrm{Sp}(1)=\SU(2)$-structure in the sense of Conti-Salamon, as noticed in \cite{Alb2015a}. It is to these and other similar structures that this article is devoted.

The following 1-form is an important irreducible tensor:
\begin{equation}
 \rho=\frac{1}{s}\,U\lrcorner\pi^\star\ric=R_{1012}e^4-R_{2012}e^3.
\end{equation}
As complex line bundles, $H_0$ and $V_0$ are very particular to dimension 3.  $V_0$ is the holomorphic tangent bundle when restricted to each fibre, $S^2$, with $\alpha_2$ restricting to the K\"ahler class. We have global 1-forms defined by
\begin{equation}\label{osrhostodos}
 \begin{split}
  \rho             =R_{1012}e^4-R_{2012}e^3 , & \\ 
  \rho_1=\rho B    =R_{1012}e^2-R_{2012}e^1 , & \\
  \rho_2=\rho I_+B =R_{1012}e^1+R_{2012}e^2 , & \\
  \rho_3=\rho I_+  =R_{1012}e^3+R_{2012}e^4 . & 
 \end{split}
\end{equation}

Now, regarding exterior derivatives, from the general formulae in \eqref{dalphanen-1} and recalling $r=R_{1010}+R_{2020}$, we have
\begin{equation}
 \dx\alpha_0 = \frac{1}{s^2}\,\theta\wedge\alpha_1,  \label{derivadasdastres2formas_alpha2} 
\end{equation}
\begin{equation}
 \dx\alpha_1 = \frac{2}{s^2}\,\theta\wedge\alpha_2-r\,\theta\wedge\alpha_0 .
\label{derivadasdastres2formas_alpha1}
\end{equation}
These are already decomposed into irreducibles. From \cite[Theorem 2.2]{Alb2015a} we have
\begin{equation}\label{dalpha0composto}
 \dx\alpha_2 = \theta\wedge\gamma-\frac{r}{2}\,\theta\wedge\alpha_1+s\,\alpha_0\wedge\rho   \quad \in\quad *W_3\oplus\llbracket*\alpha_1\rrbracket\oplus*W_2
\end{equation}
where, by \eqref{The_f_forms}, the 2-form $\gamma$ is defined as
\begin{equation}
 \gamma:=R_{1002}\psi_2+\frac{1}{2}(R_{1001}-R_{2002})\psi_1\ \ \in\  W_3.
\end{equation}
The following result shall play a relevant role later on.
\begin{prop}[\cite{Alb2015a}]\label{casocsc}
The following assertions are equivalent on a connected 3-manifold: $M$ has constant sectional curvature; $r$ is constant; $\rho=0$; $\gamma=0$; $\dx\alpha_2=-\frac{r}{2}\,\theta\wedge\alpha_1$.
\end{prop}

\subsection{Conti-Salamon structures and hypo and nearly-hypo 5-manifolds}
\label{Hanh5m}

$\SU(2)$-structures on 5 dimensions are understood as the induced metric structures on real hypersurfaces of $\SU(3)$ manifolds.

Let $S$ denote any 5-dimensional manifold. It is said that $S$ is endowed with an \textit{$\SU(2)$-structure} if its frame bundle admits a reduction to $\SU(2)=\mathrm{Sp}(1)$ via the canonical plus trivial representation in $\C^2\oplus\R$. Then there exists on $S$ an orientation and a metric such that $TS=L\oplus L^\perp$, where $L\subset TS$ is a real line bundle and $L^\perp$ is endowed with a metric compatible quaternionic structure. The concept was first introduced by D.~Conti and S.~Salamon in \cite{ContiSalamon}. Due to a canonical inclusion of $\SU(2)$ in $\SO(5)$ and lift into $\Spin(5)$, the manifold $S$ must be orientable and spin. $\SU(2)$-structures on $S$ are in one-to-one correspondence with pairs of spin structures and unit spinors.

Still following \cite{ContiSalamon}, an $\SU(2)$-structure is defined by a 1-form ${\theta}$, such that $L^\perp=\ker{\theta}$, and three 2-forms $\omega_1,\omega_2,\omega_3$ on $S$ such that
\begin{equation}\label{hypostruceq1}
\begin{split}
&\qquad\qquad\quad  {\theta}\wedge\omega_1\wedge\omega_1\neq0,\\
&\qquad\quad\quad  \omega_i\wedge\omega_j=0,\ \forall i\neq j, \\
& 2v\stackrel{\mathrm{def.}}= \omega_1\wedge\omega_1=\omega_2\wedge\omega_2= \omega_3\wedge\omega_3\neq0
\end{split}
\end{equation}
and
\begin{equation}\label{hypostruceq2}
 x\lrcorner\omega_1=y\lrcorner\omega_2\ \Longrightarrow\ \omega_3(x,y)\geq0,\ \ \forall x,y\in TS.
\end{equation}
We remark the system is verified under multiplication by any $\exp{(t\sqrt{-1})},\ t\in\R$, either on $\omega_1+\sqrt{-1}\omega_2$ or on $\omega_2+\sqrt{-1}\omega_3$.

One finds three almost complex structures $\Phi_i$ on $L^\perp$ compatible and positively tame by the respective $\omega_i$, yet inducing a unique positive definite metric
\begin{equation}\label{themetriconLperpintheory}
 \gsudois{x}{y}=\omega_i(x,\Phi_iy),\ \forall x,y\in L^\perp.
\end{equation}
Now we have a linear algebra result, which gives a formula for the induced metric without finding any $\Phi_i$.
\begin{teo}\label{Teo_themetriconLperp}
 A system $({\theta},\omega_1,\omega_2,\omega_3)$ defines an $\SU(2)$-structure on $S$ if and only if it satisfies \eqref{hypostruceq1} and the bilinear map $g_{_{\SU(2)}}$ on $L^\perp=\ker{\theta}$ given by
 \begin{equation}\label{themetriconLperp}
  x\lrcorner\omega_1\wedge y\lrcorner\omega_2\wedge\omega_3=\gsudois{x}{y}\,v, \ \ \ \forall  x,y\in L^\perp,
 \end{equation}
 is positive definite.
\end{teo}
\begin{proof}
We first prove the condition is necessary. By \cite[Corollary 1.3]{ContiSalamon} we see the triplet of the $\omega_i$ forms a frame of self-dual 2-forms of the pair $L^\perp,v$, i.e. there exists an orthonormal frame such that $\omega_1=e^{12}+e^{34},\ \omega_2=e^{13}+e^{42},\ \omega_3=e^{14}+e^{23}$. Writing $x=\sum_ix_ie_i$ and $y=\sum_iy_ie_i$, we have $x\lrcorner\omega_1=x_1e^2-x_2e^1+x_3e^4-x_4e^3$ and $y\lrcorner\omega_2=y_1e^3-y_3e^1+y_4e^2-y_2e^4$ and hence
\begin{eqnarray*}
 x\lrcorner\omega_1\wedge y\lrcorner\omega_2 &=& x_1y_1e^{23}+x_1y_3e^{12}-x_1y_2e^{24}-x_2y_1e^{13}-x_2y_4e^{12}+x_2y_2e^{14} \\ 
 & & -x_3y_1e^{34}+x_3y_3e^{14}-x_3y_4e^{24}-x_4y_3e^{13}+x_4y_4e^{23}+x_4y_2e^{34} .
\end{eqnarray*}
The identity of the given bilinear map with the metric follows:
\begin{eqnarray*}
 x\lrcorner\omega_1\wedge y\lrcorner\omega_2\wedge\omega_3 &=& (x_1y_1 +x_2y_2+x_3y_3+x_4y_4)e^{1234} .
\end{eqnarray*}

Now let us prove the condition is sufficient. Given the 1- and 2-forms satisfying \eqref{hypostruceq1}, we see there exists a four dimensional sub-vector bundle, say $L^\perp\subset TS$, on which $v$ is non-degenerate. By hypothesis, after symmetrizing, we have a positive definite metric on $L^\perp$, so it is a matter of counting dimensions to see this is uniquely determined. Since we have a volume 4-form, we know a priori that the  reduction is in 12 dimensions, from the structure group $\mathrm{SL(4)}$ to $\SU(2)$. Now the three 2-forms are written $\omega_i=\sum_{1\leq j<k\leq 4}\omega_{ijk}e^{jk}$. With \eqref{hypostruceq1} we find the $18-6=12$ dimensions. On the other hand, in order to find an orthonormal frame $e_1,\ldots,e_4$, with which one proves the 2-forms to be self-dual for the metric in \eqref{themetriconLperp}, we solve $4+3+2+1=10$ equations. So it is possible to solve these equations, leaving 2 dimensions free due to the remark above.
\end{proof}
Thus the open condition \eqref{hypostruceq2} stands for some choice of ordering of $\omega_1,\omega_2,\omega_3$.

Well understood, we assume $\|{\theta}\|_{\SU(2)}=1$ and, on the other hand, that any non-vanishing multiple of ${\theta}$ will define another $\SU(2)$-structure.

The first property of such a metric is the relation with $\SU(3)$ real submanifold geometry described in the founding article. The above notion is again equivalent to a real, oriented 5-manifold $S$, endowed with a 1-form ${\theta}$, a 2-form $\omega_1$ and a complex 2-form $\phi$, corresponding to $\omega_2+\sqrt{-1}\omega_3$, which is type $(2,0)$ for $\omega_1$, cf. \cite{ContiSalamon}, and which satisfies
\begin{equation}\label{hypostruceq3}
 \theta\wedge\omega_1\wedge\omega_1\neq0,\qquad\omega_1\wedge\phi=0,\qquad\phi\wedge\phi=0,\qquad2\omega_1\wedge\omega_1=\phi\wedge\overline{\phi}.
\end{equation}
Reciprocally, since $S$ is oriented, the corresponding $\SU(3)$-structure on $S\times\R$ follows as the pair of a real symplectic 2-form $\omega_1+\theta\wedge\dx t$ and a complex volume 3-form $\phi\wedge(\theta+\sqrt{-1}\dx t)$.

Let us now recall some further developments from \cite{BedulliVezzoni,CFS,ContiSalamon,dAFFU,FIMU} on the theory of $\SU(2)$-structures. Conserved tensors may appear, leading to the characterization of some special Riemannian geometries. Such is the case of \textit{hypo} structures, considered first in \cite{ContiSalamon}:
\begin{equation}\label{eqhypo}
 \dx\omega_1=0,\qquad\dx(\theta\wedge\omega_2)=0,\qquad\dx(\theta\wedge\omega_3)=0.
\end{equation}
There it is proved that hypo structures are precisely the $\SU(2)$-structures which are induced on a real analytic hypersurface from a complex 3-manifold endowed with an integrable $\SU(3)$-structure; `precisely' meaning that any real analytic $\SU(2)$-structure satisfying \eqref{eqhypo} arises from such a 3-fold.
\nopagebreak
\begin{figure}
\begin{center}
 \includegraphics[height=0.19\textheight]{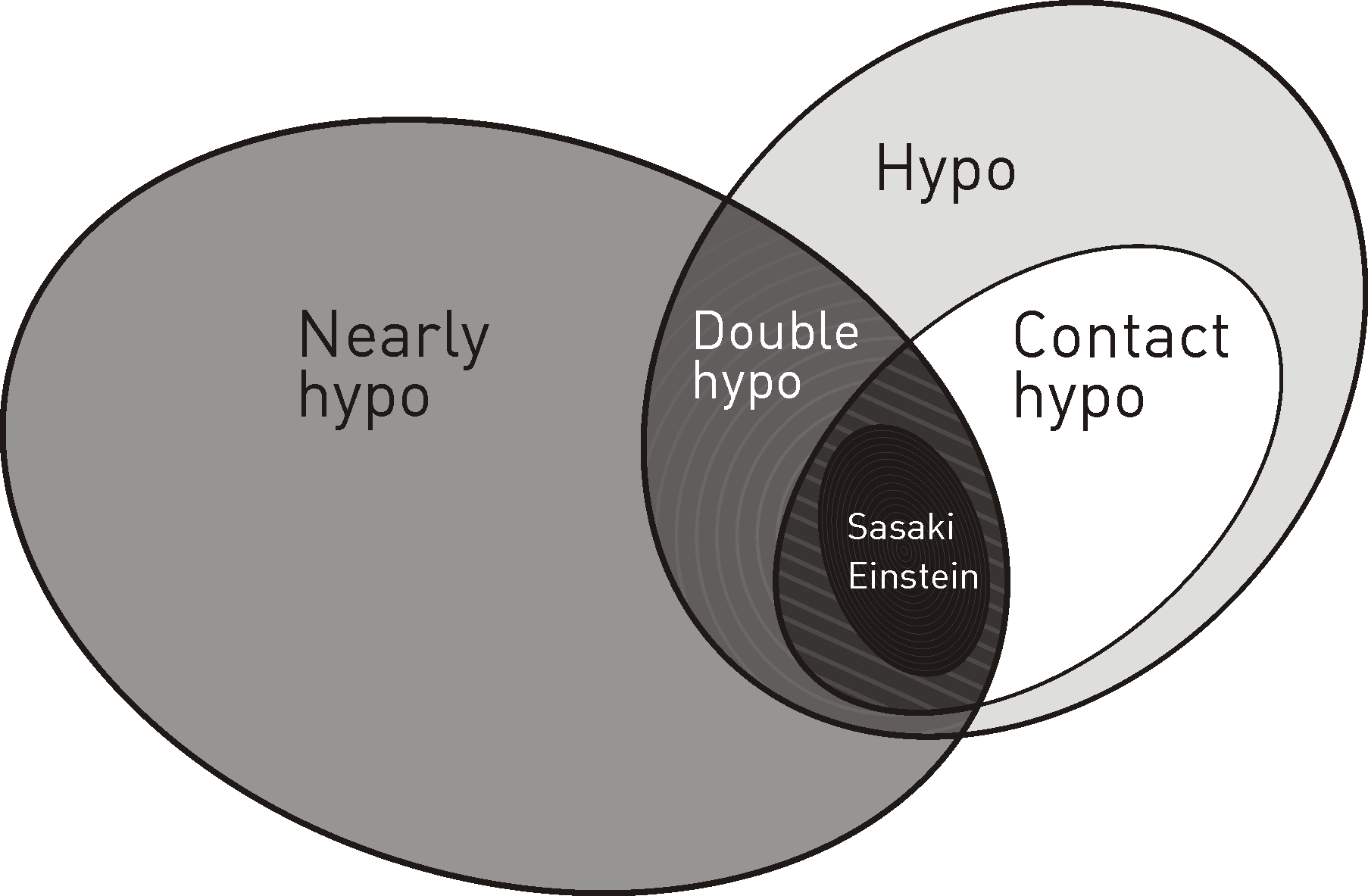}
\end{center}
\end{figure}

Another type of $\SU(2)$-structure is considered in \cite{FIMU}. The \textit{nearly-hypo} structures on a 5-manifold $S$ are defined by
\begin{equation}\label{eqnearlyhypo}
 \dx\omega_2=3\theta\wedge\omega_3,\qquad \dx(\theta\wedge\omega_1)=-2\omega_1\wedge\omega_1.
\end{equation}
Nearly-hypo structures give rise to a general construction of nearly-K\"ahler structures on $S\times\R$, cf. \cite{FIMU}. Structures which are both hypo and nearly-hypo are called \textit{double-hypo}:
\begin{equation}\label{eqdoublehypo}
 \dx\omega_1=0,\qquad\dx\omega_2=3\theta\wedge\omega_3,\qquad \dx(\theta\wedge\omega_1)=-2\omega_1\wedge\omega_1,\qquad\dx(\theta\wedge\omega_2)=0.
\end{equation}
They contain a smaller subset given by the Sasaki-Einstein 5-manifolds, i.e.
\begin{equation}\label{eqSasakiEinstein}
 \dx\theta=-2\omega_1,\qquad\dx\omega_2=3\theta\wedge\omega_3,\qquad \dx\omega_3=-3\theta\wedge\omega_2.
\end{equation}
In this case the respective $\SU(3)$-fold is a K\"ahler-Einstein manifold.
Many non-trivial examples of the above special geometries are given on products of spheres and Lie groups in \cite{FIMU}. Examples on nilmanifolds are already constructed in \cite{ContiSalamon}.

A fifth special geometry is considered and studied in \cite{BedulliVezzoni,dAFFU}: the \textit{contact-hypo} structures are defined by
\begin{equation}\label{eqcontacthypo}
 \dx\theta=-2\omega_1,\qquad\dx(\theta\wedge\omega_2)=0,\qquad\dx(\theta\wedge\omega_3)=0.
\end{equation}
Clearly, contact-hypo are hypo and contain the Sasaki-Einstein structures.

We would also consider the \textit{contact-nearly-hypo} structures as those which satisfy merely $\dx\theta=-2\omega_1$ and $\dx\omega_2=3\theta\wedge\omega_3$. However, these consist of the intersection of double-hypo and contact-hypo structures. For instance we see
\begin{eqnarray*}
\dx(\theta\wedge\omega_2)&=&\dx\theta\wedge\omega_2-\theta\wedge\dx\omega_2 \\
  &=&-\frac{1}{2}\omega_1\wedge\omega_2-\theta\wedge\theta\wedge\omega_3 \ =\ 0.
\end{eqnarray*}

Last but not least, the invariance of equations under multiplication of $\phi$ by $\exp{(t\sqrt{-1})},\ t\in\R$, is verified in the cases of hypo, contact-hypo and Sasaki-Einstein structures.

\section{On the tangent sphere bundles of 3-manifolds}

\subsection{The natural structures}
\label{TnSU2s}

We consider again the setting from Section \ref{T3dc}, where it is given an oriented Riemannian 3-manifold $(M,g)$. We may then recall the fundamental exterior differential system defined on the associated Riemannian manifold $\cals=\cals_{s,M}$. The canonical or Sasaki metric on $\cals$, also denoted by $g$, is required by the differential system, and so it shall keep its main role in the following and be referred as the \textit{canonical Sasaki metric}. We remark the natural transformations below shall lead to truly non-trivial variations of the canonical metric, cf. Section \ref{Tme}.

Since we are interested in the contact and, in particular, the Sasakian geometry of $\cals$, we shall give the name \textit{canonical Tashiro metric} or \textit{structure} to the almost contact metric and contact structure defined on $\cals$ by
\begin{equation}\label{Tashirometric1}
 \check{g}=\frac{1}{4s^2}g,\quad\eta=-\frac{1}{2s^2}\theta,\quad\Phi=B^t-B+\frac{1}{s^2}U\otimes\theta,\quad \xi=-{2}B^tU .
\end{equation}
Of course, $\theta$ comes from \eqref{thecontactform}. This structure $(\check{g},\eta,\Phi,\xi)$ is not quite the so-called standard structure, cf. \cite[Section 9]{Blair} or \cite{Calva}, but is also convenient, for many reasons. We refer the reader to \cite{AbbKowal,DrutaRomaniucOproiu} and the references therein for other important constructions of natural almost contact structures on tangent sphere bundles.

We have the canonical orientation 
\begin{equation}\label{orientation1}
 e^{01234}\simeq\eta\wedge(\dx\eta)^2 ,
\end{equation}
\begin{equation}\label{Tashirometric2}
 \eta(\xi)=1,\qquad \eta=\xi\lrcorner\check{g}, \qquad \Phi^2=-1_{{T\cals}}+\xi\otimes\eta,
\end{equation}
\begin{equation}\label{Tashirometric3}
 \check{g}=\frac{1}{2}\dx\eta(\Phi\ ,\ ),\qquad \check{g}(\Phi\ ,\Phi\ )= \check{g}-\eta\otimes\eta 
\end{equation}
and a reciprocal curious result.
\begin{prop}\label{Prop_Tashirometric}
 If $\eta=-p\theta$ for some $p\neq0$ and we are to have (\ref{orientation1}--\ref{Tashirometric3}), then $p=1/2s^2$ and the almost contact metric structure is given by \eqref{Tashirometric1}.
\end{prop}
\begin{proof}
 Indeed the conditions are the required for an almost contact structure, cf. \cite{Boyer1,Boyer2}. Recalling the non-linear property of the geodesic 1-form $\theta$, which, for every radius $s$, satisfies $\dx\theta=g((B^t-B)\otimes 1)$, the result follows by simple computations.
\end{proof}
Notice the previous results are valid in any dimension. A classical result of Tashiro proves $\check{g}$ is Sasakian if and only if $M$ has constant sectional curvature $\frac{1}{s^2}$.

We may also deform the almost contact metric structure along $\xi$, i.e. taking for Reeb vector field a multiple $\lambda$ of $\xi$ different of that for $g$ in $\xi^\perp=H_0\oplus V_0$. It is known that for certain values of $\lambda$ this metric is Sasaki-Einstein. This phenomena shall appear below.

Finally we are in the right moment to recall the main purpose of this article. That is, to study the $\SU(2)$-structures on $\cals$ induced by the differential system $\theta,\alpha_0,\alpha_1,\alpha_2$.

First of all we are led to define the 1-form $\tilde{\theta}$ from \eqref{hypostruceq1} as a multiple of the canonical contact 1-form $\theta$; secondly we define the three 2-forms $\omega_1,\omega_2,\omega_3$ as the linear combinations
\begin{equation}\label{definitionofinvariantSU2}
\begin{split}
 \omega_1 &=a_0\alpha_0+a_1\alpha_1+a_2\alpha_2+a_3\dx\theta \\
 \omega_2 &=b_0\alpha_0+b_1\alpha_1+b_2\alpha_2+b_3\dx\theta \\
 \omega_3 &=c_0\alpha_0+c_1\alpha_1+c_2\alpha_2+c_3\dx\theta 
\end{split}
 \end{equation}
where $a_0,\ldots,a_3,b_0,\ldots,b_3,c_0,\ldots,c_3$ are constant coefficients.

Notice we have $\|\theta\|=s$, but always $\|\tilde{\theta}\|_{\SU(2)}=1$. Also, due to \eqref{basicstructurequations}, we have
\begin{equation}
 \omega_1\wedge\omega_1=(a_1^2+a_3^2-a_0a_2)\,\dx\theta\wedge\dx\theta,
\end{equation}
\begin{equation}
 \omega_1\wedge\omega_2=(a_1b_1+a_3b_3-\frac{1}{2}a_0b_2-\frac{1}{2}a_2b_0)\,\dx\theta\wedge\dx\theta
\end{equation}
and similar identities with the $a,b,c$'s. Therefore, by \eqref{hypostruceq1}, we must have 
\begin{equation}\label{hypostruceq4}
  a_1^2+a_3^2-a_0a_2=b_1^2+b_3^2-b_0b_2=c_1^2+c_3^2-c_0c_2\neq0
\end{equation}
and
\begin{equation}\label{hypostruceq5}
\begin{split}
 & a_0b_2+a_2b_0-2a_1b_1-2a_3b_3 = b_0c_2+b_2c_0-2b_1c_1-2b_3c_3 =\qquad \qquad\\ 
 & \hspace{47mm}  = c_0a_2+c_2a_0-2c_1a_1-2c_3a_3 =0.
\end{split}
\end{equation}

\begin{defi}
 A set of differential forms $\tilde{\theta},\omega_1,\omega_2,\omega_3$ as the above defined, with constant coefficients and satisfying (\ref{hypostruceq1},\ref{hypostruceq2}), is called a \emph{natural $\SU(2)$-structure} on $\cals$.
\end{defi}

The natural $\SU(2)$-structures are natural variations of the Sasaki metric on tangent sphere bundles. They induce `g-natural' metrics in the sense of well-known references, such as \cite{Abb1,AbbKowal}, but they also give a new class of natural metrics in these 5-dimensional manifolds. Our metrics are indeed of a more general kind as we shall see in Proposition \ref{SU2metricsareofmoregeneralkind}.

\vspace{3mm}

\noindent
\textsc{Main example.}
\begin{itemize}
    \item This \textit{main example} was first devised in \cite{Alb2015a}. The orientation on $\cals_{s,M}$ is induced by the ordering of any adapted coframe $e^0,e^1,\ldots,e^4$; but that on $\ker\theta$, corresponding to $v$, is $-e^{1234}$. The $\SU(2)$-structure is given by ($\dx\tilde{\theta}=-2\omega_1$)
\begin{equation}\label{mainexample}
 \tilde{\theta}=-2\theta=-2s\,e^0,\qquad\omega_1=\dx\theta,
   \qquad\omega_2=\alpha_2-\alpha_0,\qquad\omega_3=\alpha_1 .
\end{equation}
Indeed, recalling the theory of the intrinsic geometry of Riemannian 3-manifolds, we see that $-2\alpha_0\wedge\alpha_2=\alpha_1\wedge\alpha_1= \dx\theta\wedge\dx\theta=-2e^{1234}$. Below we shall prove the $\SU(2)$-metric is Sasaki-Einstein if and only if $M$ has constant sectional curvature $K=3$ and $s=\sqrt{3}/3=\sqrt{1/K}$. Also we shall see the induced metric coincides with the canonical metric on $\cals$ if and only if $s=\frac{1}{2}$ (just because of $\tilde{\theta}$). Hence it is not the Sasakian, Tashiro metric $\check{g}$ on $\cals_{s,S^3(s)}$ which is an Einstein metric. Finally we remark the choice of $\tilde{\theta}=\theta,\omega_1=-\frac{1}{2}\dx\theta$, etc., seems equally keen in the search for hypo equations, but then we would miss the canonical Sasaki metric.
\end{itemize}

\vspace{3mm}

Notice the canonical Tashiro structure cannot be transformed homothetically into the structure of the \textit{main example}, as Proposition \ref{Prop_Tashirometric} shows, except for $s=\frac{1}{2}$.

We remark the general $\SU(2)$-structure remains invariant under isometries of $M$ lifted to isometries of the radius $s$ tangent sphere bundle total space $\cals$ with the canonical metric. Indeed, an adapted frame is transformed into an adapted frame. In particular, the new structures descend to a quotient space $\cals/\Gamma\lrr M/\Gamma$ for any discrete subgroup $\Gamma\subset\mathrm{Isom}_+(M)$.


Let us now analyse one of the structural equations.
\begin{lemma}\label{lemadomega1iguala0}
Suppose $\dx\omega_1=0$. Then:\\
(i) $\omega_1=a_3\dx\theta$ or\\
(ii) $\omega_1=a_0\alpha_0+a_2\alpha_2+a_3\dx\theta$, with $a_2\neq0$, and $M$ has constant sectional curvature $K=\frac{a_0}{a_2s^2}$.
\end{lemma}
\begin{proof}
 Since $\dx\alpha_0=\frac{1}{s^2}\,\theta\wedge\alpha_1$, $\dx\alpha_1=\frac{2}{s^2}\,\theta\wedge\alpha_2-r\,\theta\wedge\alpha_0$ and $\dx\alpha_2 = \theta\wedge\gamma-\frac{r}{2}\,\theta\wedge\alpha_1+s\,\alpha_0\wedge\rho$, it follows immediately from the hypothesis that either $a_2=0$ or $\rho=0$. Also we find $\frac{a_0}{s^2}\alpha_1+\frac{2a_1}{s^2}\alpha_2-ra_1\alpha_0+a_2\gamma-\frac{ra_2}{2}\alpha_1=0$. Knowing the representation subspaces, this implies $a_1=0$, $2a_0-ra_2s^2=0$ and $a_2\gamma=0$. Now if $a_2=0$, then $a_0=a_1=0$ and we are in case i. If $a_2\neq0$, then $\gamma=0$ and by Proposition \ref{casocsc} we have constant sectional curvature $K$ given by $a_0-Ka_2s^2=0$.
\end{proof}
Now let us study a second main equation, common to all five special $\SU(2)$-structures recalled in Section \ref{Hanh5m}. Indeed, $\theta\wedge\omega_3$ must always be closed. Let us consider real constants $c_0,\ldots,c_3$ and 
\begin{equation}
 \omega_3=c_0\alpha_0+c_1\alpha_1+c_2\alpha_2+c_3\dx\theta.
\end{equation}
\begin{lemma}\label{lemadtetaomega3}
Suppose $\dx(\theta\wedge\omega_3)=0$. Then:\\
(i) $\omega_3=c_0\alpha_0+c_1\alpha_1+c_2\alpha_2$, if $M$ has constant sectional curvature, \\
(ii) $\omega_3=c_0\alpha_0+c_1\alpha_1$, if $M$ has non-constant sectional curvature.
\end{lemma}
\begin{proof}
 We have $\dx(\theta\wedge\omega_3)=\dx\theta\wedge\omega_3-\theta\wedge(\sum c_j\dx\alpha_j)$. From the first summand and the fundamental equations \eqref{basicstructurequations} it follows that $c_3=0$; the remaining summand gives the equivalent condition that $\rho=0$ or $c_2=0$. Applying again Proposition \ref{casocsc}, we have the result.
\end{proof}
Constant coefficients restrict the curvature on the base manifold.
\begin{prop}\label{Excludingnonconstantsectionalcurvature}
Suppose $M$ has non-constant sectional curvature. Then there do not exist natural hypo nor natural nearly-hypo structures on $\cals$. 
\end{prop}
\begin{proof}
 Regarding the case hypo, by definition and part ii of the above Lemma, we would need two 2-forms $\omega_2=b_0\alpha_0+b_1\alpha_1,\ \omega_3=c_0\alpha_0+c_1\alpha_1$ satisfying the orthogonality relations $b_1^2=c_1^2\neq0$ and $b_1c_1=0$. For the case nearly-hypo, it is not possible also to have three natural 2-forms giving a nearly-hypo sphere bundle, because, in searching for $\omega_2=b_0\alpha_0+b_1\alpha_1+b_2\alpha_2+b_3\dx\theta$ satisfying \eqref{eqnearlyhypo} and in particular $\dx\omega_2=3\theta\wedge\omega_3$ for the necessarily $\omega_3=c_0\alpha_0+c_1\alpha_1$ found above, we deduce $\omega_2=b_0\alpha_0$, which has vanishing square.
\end{proof}

\subsection{Structures of type I}
\label{oftypeI}

Following the above conclusions, we assume $M$ has constant sectional curvature. A first candidate for $\omega_1$ is that which is found in case i of Lemma \ref{lemadomega1iguala0}. We thus consider $\SU(2)$-structures with
\begin{equation}\label{omega1typei}
 \omega_1=\dx\theta .
\end{equation}
\begin{Rema}
 For a generalization, if we take $\omega_1=\dx\theta$ and find a hypo structure, then the structure can be adjusted accordingly (simply multiplying $\omega_2,\omega_3$ by the same $a_3$). Notwithstanding, for the nearly-hypo equations it is different. Assuming we have found \eqref{eqnearlyhypo} for the pair $\tilde{\theta},\omega_1=\dx\theta$, then the referred variation of $\omega_1$ together with $\tilde{\tilde{\theta}}=\lambda\tilde{\theta}$, $\lambda\in\R$, yields by
 \eqref{eqnearlyhypo}
 \begin{equation}
  a_3=\lambda a_3\qquad\mbox{and}\qquad \lambda a_3=a_3^2
 \end{equation}
 implying $a_3=1$. Therefore the solutions are 1-1 dependent on $a_3$. The study then continues in the next section.
\end{Rema}

We shall have a hypo structure and, preferably, a contact-hypo structure, if we let $\tilde{\theta}=-2\theta$ and take any two 2-forms, deduced from case i of Lemma \ref{lemadtetaomega3}, satisfying \eqref{hypostruceq1} and \eqref{themetriconLperp}
\begin{equation}\label{omega2and3foromega1typei}
\omega_2=b_0\alpha_0+b_1\alpha_1+b_2\alpha_2,\qquad \omega_3=c_0\alpha_0+c_1\alpha_1+c_2\alpha_2 .
\end{equation}
These shall be called the $\SU(2)$-structures of type I. In sum, as in (\ref{hypostruceq4},\ref{hypostruceq5}), we find the system
\begin{equation}\label{beesecees}
 \begin{cases}
  b_1^2-b_0b_2=1 \\ 
  c_1^2-c_0c_2=1 \\ 
  b_0c_2+b_2c_0-2b_1c_1=0.
 \end{cases}
\end{equation}
A last condition is to be fulfilled by the $b_i,c_i\in\R$: that $\phi=\omega_2+\sqrt{-1}\omega_3$ is $(2,0)$ for $\omega_1$, cf. \eqref{hypostruceq2}. As expected, notice the symmetry $\phi\rightsquigarrow \exp({\sqrt{-1}t})\phi$ leaves the system \eqref{beesecees} invariant.

\begin{prop}[\bf $\SU(2)$-structures of type I]\label{ProptheinvariantSU2structures}
The natural $\SU(2)$-structures on $\cals$ given by the canonical contact 1-form $\tilde{\theta}$ and by the 2-forms $\omega_1,\omega_2,\omega_3$ in \eqref{omega1typei},\eqref{omega2and3foromega1typei} and \eqref{beesecees} are in one-to-one correspondence with points of the real hypersurface
\begin{equation}\label{planopseudoesfera}
 \bigl\{(X,Y,A,B)\in\R^4:\ B^2(1+A^2)^2(X^2+Y^2)=1,\,\ B>0\bigr\},
\end{equation}
via the transformation
\begin{equation}\label{coeficientesSU2estrutura}
\begin{cases}
 b_0=(1-A^2)B^2X+2AB^2Y \\
 b_1=(1+A^2)B(Y-AX)  \\
 b_2=-(1+A^2)^2X
\end{cases}  \qquad
\begin{cases}
 c_0=(1-A^2)B^2Y-2AB^2X \\
 c_1=-(1+A^2)B(X+AY)  \\
 c_2=-(1+A^2)^2Y
\end{cases}.
\end{equation}
\end{prop}
\begin{proof}
 Let $e_0,e_1,e_2,e_3,e_4$ be an adapted frame orthonormal for the canonical metric. Since $e_0$ is in the annihilator of all $\omega_i$, it follows the new metric on $\cals$ will have $e_0$ orthogonal to the remaining $e_j$. Since the structure is invariant, the compatible almost complex structures $\Phi_i$ on $\ker\theta$ will be invariant (by isometries of $M$ lifted to $\cals$). For $\Phi_1$ compatible with $\omega_1=\dx\theta$ and respecting formula \eqref{themetriconLperpintheory}, we may hence write $\Phi_1x^h=Ax^h-Bx^v$ with some constants $A,B$ and $B>0$, where $x$ is any vector on $T_{\pi(u)}M$ orthogonal to $u\in\cals$ and $x^h,x^v$ are the canonical lifts. The space of $\Phi_1$ is indeed determined completely by $A$ and $B$ (it agrees with the symmetric space $\mathrm{Sp}(2,\R)/\Uni(1)$, the Siegel domain or Poincar\'e half-plane, as studied e.g. in \cite{AlbRaw}). Thus a basis $\{\beta_1,\beta_2\}$ of $(1,0)$-forms is determined up to factors by 
 \[ \beta_1=e^1+\sqrt{-1}(\lambda e^3+\mu e^1)\mod\R ,  \]
such that
\[  \beta_1(e_1+\sqrt{-1}\Phi_1e_1)=0 ,  \]
and similarly for $\beta_2$ recurring to the mirror pair $e_2,e_4$. Solving for $\lambda,\mu$ and removing denominators, we obtain explicit solutions:
\[ \begin{cases}  \beta_1=-Be^1+\sqrt{-1}(ABe^1+(1+A^2)e^3) \\
 \beta_2=-Be^2+\sqrt{-1}(ABe^2+(1+A^2)e^4)    \end{cases}  .\]
The $(2,0)$-form $\beta_1\wedge\beta_2$ is independent of the adapted frame, as expected:
\begin{eqnarray*}
 \beta_1\wedge\beta_2 &=& (B^2-A^2B^2)e^{12}-AB(1+A^2)(e^{32}+e^{14})-(1+A^2)^2e^{34} + \\
  & &\qquad+\sqrt{-1}\bigl(-AB^2e^{12} 
  -B(1+A^2)e^{14}-AB^2e^{12}-B(1+A^2)e^{32}\bigr) \\
  &=& B^2(1-A^2)\alpha_0-AB(1+A^2)\alpha_1-(1+A^2)^2\alpha_2+ \\
  & &  \qquad\qquad +\sqrt{-1}\bigl(-2AB^2\alpha_0-B(1+A^2)\alpha_1\bigr).
\end{eqnarray*}
The last condition required by an $\SU(2)$-structure is that $\omega_2+\sqrt{-1}\omega_3$ is a form of type $(2,0)$-for $\Phi_1$. In other words, we must have $\omega_2+\sqrt{-1}\omega_3=(X+\sqrt{-1}Y)\,\beta_1\wedge\beta_2$ for some $X,Y\in\R$. Equivalently,
\begin{equation*}
\begin{split}
& \qquad\qquad b_0\alpha_0+b_1\alpha_1+b_2\alpha_2+\sqrt{-1}(c_0\alpha_0+c_1\alpha_1+c_2\alpha_2)= \\
& =XB^2(1-A^2)\alpha_0-XAB(1+A^2)\alpha_1-X(1+A^2)^2\alpha_2
  +2YAB^2\alpha_0+YB(1+A^2)\alpha_1  +\\
&  \sqrt{-1}\bigl(-2XAB^2\alpha_0-XB(1+A^2)\alpha_1+YB^2(1-A^2)\alpha_0 -YAB(1+A^2)\alpha_1-Y(1+A^2)^2\alpha_2\bigr).
\end{split}
\end{equation*}
This yields formulae \eqref{coeficientesSU2estrutura} for the coefficients $b_0,\ldots,b_2,c_0,\ldots,c_2$. Recalling \eqref{beesecees}, then two short computations on the first rows, $b_1^2-b_0b_2=1$ and $c_1^2-c_0c_2=1$, yield the very same condition which is that defining the set \eqref{planopseudoesfera}. Finally, the last equation is automatically satisfied, as we care to show next. Indeed, we have $b_2c_0-c_1b_1=A$ and $b_0c_2-b_1c_1=-A$. Let us see this last identity:
\begin{equation*}
 \begin{split}
  b_0c_2-b_1c_1&=(1+A^2)^2B^2\bigl(-(1-A^2)XY-2AY^2 
  +(Y-XA)(X+YA)\bigr) \\
 &= (1+A^2)^2B^2\bigl(-XY+A^2XY-2AY^2+XY+AY^2-AX^2-A^2XY\bigr)  \\
 &= (1+A^2)^2AB^2(-Y^2-X^2) \\
 &=-A.
 \end{split}
\end{equation*}
Hence $b_0c_2+b_2c_0-2b_1c_1=0$.
\end{proof}
The above Proposition characterizes completely the 3-dimensional family of natural $\SU(2)$-structures of type I. Later we shall see that condition \eqref{hypostruceq2} is assured by 
\begin{equation}\label{typeImetricpositivedefinite}
 b_1c_0-b_0c_1>0.
\end{equation}
The next result shall also be duely proved in Section \ref{Atcomega_1igualadteta}.
\begin{prop}\label{Propositioncompatiblemetric}
 The $\SU(2)$-structures of type I which are \emph{compatible} with the canonical metric are given by $A=0,\ B=1,\ X^2+Y^2=1$. 
\end{prop}
Recall the set of three 2-forms on the radius $s$ tangent manifold $\cals$ determines the Riemannian structure up to the fixed $\|\tilde{\theta}\|_{\SU(2)}=1$  (whereas $\|\theta\|=s$). Hence the meaning of the word \textit{compatible} in the last Proposition: the precisely same metric on $\ker\theta$.

We now state the result which follows from various remarks above.
\begin{teo}[\bf Hypo]\label{Teo_hypo}
 A natural $\SU(2)$-structure on $\cals$ with $\omega_1=\dx\theta$ is hypo if and only if $M$ has constant sectional curvature and it is of type I. Defining $\tilde{\theta}=-2\theta$ we obtain a contact-hypo structure, i.e. satisfying also $\dx\tilde{\theta}=-2\omega_1$. 
 
 Moreover, for any $X,Y\in\R$ such that $X^2+Y^2=1$, the $\SU(2)$-structure given by
 \begin{equation}
  \omega_2=X\alpha_0+Y\alpha_1-X\alpha_2,\qquad\omega_3=Y\alpha_0-X\alpha_1-Y\alpha_2
 \end{equation}
is hypo and compatible with the canonical metric.
\end{teo}
\begin{coro}
 For any oriented Riemannian 3-manifold $M$, the \emph{main example}, \eqref{mainexample}, defines a contact $\SU(2)$-structure compatible with the canonical metric; which is hypo if and only if $M$ has constant sectional curvature.
\end{coro}
Thus, for each pair $K,s$, there exists a 3 dimensional family of contact-hypo structures. However, notice that, as it happens with the \textit{main example}, the induced metric is the same under symmetry $\phi\rightsquigarrow \exp({\sqrt{-1}t})\phi$.

Let us now find the natural nearly-hypo structures, still with the obvious $\omega_1$. Let us stress that we exclude non-constant sectional curvature due to Proposition \ref{Excludingnonconstantsectionalcurvature}.
\begin{teo}[\bf Nearly-hypo]\label{Teo_nearlyhypo}
 Suppose $M$ has constant sectional curvature $K$. Then the natural $\SU(2)$-structures on the radius $s$ tangent sphere bundle total space $\cals$, with $\tilde{\theta}=-2\theta$ and $\omega_1=\dx\theta$, are nearly-hypo if and only if they are of the kind given in Proposition \ref{ProptheinvariantSU2structures} and, moreover, of the kind given by
 \begin{equation}\label{S_equationsnearlyhypo1}
  \omega_2=b_0\alpha_0+b_1\alpha_1+b_2\alpha_2,\qquad
 \omega_3=\frac{Kb_1}{3}\alpha_0
     +\frac{s^2Kb_2-b_0}{6s^2}\alpha_1-\frac{b_1}{3s^2}\alpha_2
\end{equation}
for any $b_0,b_1,b_2\in\R$ such that $b_1^2-b_0b_2=1$ and
 \begin{equation}\label{S_equationsnearlyhypo2}
    (b_0+s^2Kb_2)^2+4s^2K=36s^4 
 \end{equation}
and
\begin{equation}
 K>-\frac{b_0^2}{s^2(1+b_1^2)}.
\end{equation}
Moreover, such nearly-hypo structures are always contact-hypo. 

The structures are compatible with the canonical metric if and only if (i) $b_2=-b_0$, $b_1\neq0$, $b_0^2+b_1^2=1$, $K=3=s^{-2}$, or (ii) $b_2=-b_0=\pm1$, $b_1=0$, $s^2K+1=6s^2$.
\end{teo}
\begin{proof}
 By condition $\omega_1\wedge\omega_2=0$, we must have $\omega_2= b_0\alpha_0+b_1\alpha_1+b_2\alpha_2$, and by the same reason or from Lemma \ref{lemadtetaomega3}, we must have $\omega_3=c_0\alpha_0+c_1\alpha_1+c_2\alpha_2$. Hence such nearly-hypo structures exist if and only if they are of the referred kind, this is, type I or \eqref{beesecees}. Next, we see that we just have to study $\dx\omega_2=3\tilde{\theta}\wedge\omega_3$. Knowing that the Ricci curvature of $M$ satisfies $r=2K$, we obtain the formula for $\omega_3$:
 \begin{eqnarray*}
  \dx\omega_2 &=& b_0\dx\alpha_0+b_1\dx\alpha_1+b_2\dx\alpha_2 \\
     &=& \theta\wedge(\frac{b_0}{s^2}\alpha_1+\frac{2b_1}{s^2}\alpha_2-rb_1\alpha_0-\frac{r}{2}b_2\alpha_1)\\
     &=& \theta\wedge\bigl(-2Kb_1\alpha_0
     +\frac{b_0-s^2Kb_2}{s^2}\alpha_1+\frac{2b_1}{s^2}\alpha_2\bigr) \\
     &=& 3\tilde{\theta}\wedge\bigl(\frac{Kb_1}{3}\alpha_0
     +\frac{s^2Kb_2-b_0}{6s^2}\alpha_1-\frac{b_1}{3s^2}\alpha_2\bigr).
 \end{eqnarray*}
A computation on $c_1^2-c_0c_2=1$ yields $(b_0+s^2Kb_2)^2+4s^2K=36s^4$, and these conditions together with \eqref{typeImetricpositivedefinite} are sufficient. Indeed, a very surprising result, the remaining equation is immediately satisfied:
\[ b_0c_2+b_2c_0-2b_1c_1=-\frac{b_0b_1}{3s^2}+\frac{Kb_2b_1}{3}-\frac{2Kb_1b_2}{6}+\frac{2b_0b_1}{6s^2}=0 . \]
It is trivial to prove that $\dx(\tilde{\theta}\wedge\omega_2)=0$. Indirectly, we note the structure is contact-nearly-hypo, cf. ending of Section \ref{Hanh5m}. Hence it is double-hypo.

Compatibility with the canonical metric is easily seen to be equivalent to cases i or ii. Only $c_1^2-c_0c_2=1$ needs verification: in case i we have
\[ (b_0+s^2Kb_2)^2+4s^2K=(b_0+b_2)^2+4=36/9=36s^4\]
while case ii is
\[  (b_0+s^2Kb_2)^2+4s^2K=(2-6s^2)^2+4(6s^2-1)=4-24s^2+36s^4+24s^2-4=36s^4 \]
as we wished.
\end{proof}

\vspace{2mm}

\noindent
\textsc{Some examples.}

\begin{itemize}
   \item It seems there should exist a $(c_0,c_1,c_3)$ \textit{conjugate} to the class of solutions $(b_0,b_1,b_2)=(b_0,b_0+1,b_0+2)$, for any $b_0\in\R$, of $b_1^2-b_0b_2=1$. Looking at $\omega_3$ above, then the best answer might always depend on $K$. Also notice this example and case ii above both contain the \textit{main example}, $b_0=-1$.
   \item Let us see the flat case, $K=0$. The product manifold $\R^3\times S^2(s_0)$ for $s_0=\sqrt{1/6}$  admits two, the author believes non-isometric, $\SU(2)$-structures both contact-hypo and double-hypo and not Sasaki-Einstein. The first is the \textit{main example}. The second is the above, necessarily with $b_0^2=36s_0^4=1$. We chose $s_0$ on purpose, because we may then have $b_0=-1$, which indeed returns to the \textit{main example}. But also we may have $b_0=1$ and then find a structure given by $\tilde{\theta}=-2\theta,\ \omega_1=\dx\theta$,
\begin{equation}\label{contacdoublehypoR3xS2nonSE}
 \omega_2=\alpha_0+2\alpha_1+3\alpha_2,\quad\omega_3=-\alpha_1-4\alpha_2.
\end{equation}
  
   \item For $M$ a hyperbolic space we may also consider the \textit{main example}, case ii, to find another interesting double-hypo structure. For example, letting $K=-3$ and $s=\frac{1}{3}$, the required inequality holds. We remark that in this case $\dx\omega_3=\dx\alpha_1=-3\tilde{\theta}\wedge(3\alpha_2+\alpha_0)$.

\end{itemize}

Thus, for each pair $K,s$, there exists a 1-dimensional family of nearly-hypo structures. Now let us see the conditions for the Sasaki-Einstein structures.
\begin{coro}\label{Coro_S_SasakiEinsteincurvaturaebees}
 The double-hypo structures in Theorem \ref{Teo_nearlyhypo} are Sasaki-Einstein if and only if $M$ has positive constant sectional curvature
 \begin{equation}\label{S_SasakiEinsteincurvaturaebees}
  K=9s^2 .
 \end{equation}
In particular, of the double-hypo structures compatible with the canonical metric, case i is always Sasaki-Einstein, while case ii implies $K=3=s^{-2}$ --- which is i again.
\end{coro}
\begin{proof}
 The condition to be verified is just $\dx\omega_3=-3\tilde{\theta}\wedge\omega_2=6\theta\wedge\omega_2$ where $\omega_2,\omega_3$ are given by the Theorem. On the left hand side we have
 \begin{eqnarray*}
  \dx\omega_3 
     &=& \frac{Kb_1}{3s^2}\theta\wedge\alpha_1+\frac{s^2Kb_2-b_0}{6s^2} \bigl(\frac{2}{s^2}\theta\wedge\alpha_2-2K\theta\wedge\alpha_0\bigr) +\frac{b_1}{3s^2}K\theta\wedge\alpha_1 \\
   &=& 6\theta\wedge\bigl(\frac{Kb_0-s^2K^2b_2}{18s^2}\alpha_0 
   +\frac{Kb_1}{9s^2}\alpha_1+\frac{s^2Kb_2-b_0}{18s^4}\alpha_2\bigr)
 \end{eqnarray*}
and so
\[ Kb_0-s^2K^2b_2=18s^2b_0,\qquad Kb_1=9s^2b_1,
\qquad s^2Kb_2-b_0=18s^4b_2.  \]
For $b_1\neq0$,
\[ 9s^2b_0-81s^6b_2=18s^2b_0,\qquad K=9s^2,
\qquad 9s^4b_2-b_0=18s^4b_2 .  \]
The first and the last equations are, respectively, $-9s^4b_2=b_0,\ -b_0=9s^4b_2$. But these are both equivalent to $b_0+s^2Kb_2=0$, precisely the condition in \eqref{S_equationsnearlyhypo2}. For $b_1=0$, we have $b_0b_2=-1$, and then we see the remaining two equations yield $K+s^2K^2b_2^2=18s^2$ and $s^2Kb_2^2+1=18s^4b_2^2$ (multiplying by $b_0$ gives equivalent conditions). These two imply $K=9s^2$ and so we may proceed as before.

Finally, case i in the Theorem clearly satisfies $K=3=9s^2$. Case ii yields the very same condition, because the solution to $K=\frac{6s^2-1}{s^2}=9s^2$ is precisely $s^2=\frac{1}{3}$ and $K=3$.
\end{proof}

We describe all natural Sasaki-Einstein structures on $\cals$ with 
\begin{equation}
 \tilde{\theta}=-2\theta \ \qquad \mbox{and} \ \qquad \omega_1=\dx\theta.
\end{equation}
Since $K=9s^2,\ b_1^2=1+b_0b_2$ and $b_0+s^2Kb_2=0$, we  define $Q=Q(s,b_2)=\pm\sqrt{1-9s^4b_2^2}$. Then the two remaining 2-forms satisfying \eqref{eqSasakiEinstein} are
\begin{equation}\label{S_equationsSasakiEinstein}
\begin{split}
\omega_2=-9s^4b_2\alpha_0+Q\alpha_1+b_2\alpha_2, & \\
 \omega_3=3s^2Q\alpha_0+3s^2b_2\alpha_1-\frac{Q}{3s^2}\alpha_2 .&
 \end{split}
\end{equation}
Below we shall find more information on the metric: it is the same for all $b_2$. Actually this symmetry is the natural invariance on $\exp(t\sqrt{-1})(\omega_2+\sqrt{-1}\omega_3)$.

\vspace{2mm}

\noindent
\textsc{Some examples.}

\begin{itemize}
   \item Assuming $Q=0$ (one can also follow $b_2=0$ for this case), equivalently, $b_2=\pm\frac{1}{3s^2}$, we have
\begin{equation}
    \omega_2= \mp3s^2\alpha_0\pm\frac{1}{3s^2}\alpha_2,\qquad \omega_3=\pm\alpha_1 .
\end{equation}
    In particular, for $\cals_{s,M}$ with ray $s=\sqrt{3}/3$ we obtain the \textit{main example}, \eqref{mainexample}.
    \item By an exact sequence of homotopy groups, the simply connected Sasaki-Einstein structures compatible with the canonical metric are given over a unique simply connected base of sectional curvature $K=3$ and tangent sphere bundle with radius $s=\sqrt{3}/3$. This is, precisely the sphere $M=S^3(s)$ since $K=1/s^2$. The condition of equal radius on both base and tangent spheres, in the quest for a \textit{Sasakian} manifold, was first found by Tashiro, cf. \cite{Alb2011arxiv,Alb2015a}. The present metric is different.
\end{itemize}

Our invariant theory, as mentioned earlier, is suitable for any quotient manifold $M/\Gamma$ where $\Gamma$ is a discrete group of isometries. New Sasaki-Einstein metrics on the product of $S^2$ with a lens space may hence be described. We recall that such metrics on such products were found in \cite{GMSW}, with a particular interest on 3-dimensional lens spaces; a coincidence with the metrics above is therefore not to be excluded.

\subsection{Other hypo and nearly-hypo structures and case ii of Lemma \ref{lemadomega1iguala0}}

Let us return to the general construction in Section \ref{TnSU2s}. We may search for natural nearly-hypo structures with $\tilde{\theta}=-2p\theta,\ p\neq0$, and generic $\omega_1$ different from the above. Easy enough, equation $\dx(\tilde{\theta}\wedge\omega_1)=-2\omega_1\wedge\omega_1$ is equivalent to constant sectional curvature of $M$ and 
\begin{equation}
 a_3p=a_1^2+a_3^2-a_0a_2 . 
\end{equation}
The $\SU(2)$-structure requires $a_3\neq0$. Now, given a pair of generic 2-forms $\omega_2,\omega_3$ such that $\dx\omega_2=3\tilde{\theta}\wedge\omega_3=-6p\theta\wedge\omega_3$, then, recalling the computation in the proof of Theorem \ref{Teo_nearlyhypo}, we see immediately how to write $\omega_3$ in terms of the coefficients of $\omega_2=b_0\alpha_0+b_1\alpha_1+b_2\alpha_2+b_3\dx\theta$:
 \begin{equation}\label{omega3ofcertainlynearlyhypo}
 \omega_3=\frac{Kb_1}{3p}\alpha_0+
 \frac{s^2Kb_2-b_0}{6s^2p}\alpha_1-\frac{b_1}{3s^2p}\alpha_2 .
\end{equation}
In particular, as found much earlier, we must have $c_3=0$. The solutions for a nearly-hypo structure are thus found within the following system, cf. \eqref{hypostruceq1}:
 \begin{equation}\label{generalnearlyhypo}
  \begin{cases}
   a_1^2+a_3^2-a_0a_2=a_3p \\
   b_1^2+b_3^2-b_0b_2=a_3p \\ b_0^2-2s^2Kb_0b_2+s^4K^2b_2^2+4s^2Kb_1^2=36s^4a_3p^3 \\  
   a_0b_2+a_2b_0-2a_1b_1-2a_3b_3=0 \\  
   a_0b_1-s^2Ka_2b_1+s^2Ka_1b_2-a_1b_0=0 .
  \end{cases}
 \end{equation}
A sixth equation would come from $\omega_2\wedge\omega_3=0$, but one sees this is automatically satisfied --- `a very surprising' result already seen above.

Clearly, even the case $a_0=a_1=a_2=0$ is difficult to study.

Now let us look again for hypo structures, just satisfying $\dx\omega_1=0$.
We are led to case ii of Lemma \ref{lemadomega1iguala0}, necessarily on a base $M$ of constant sectional curva\-tu\-re $K=\frac{a_0}{a_2s^2}$, where $a_2\neq0$, and a closed 2-form, necessarily with $a_1=0$,
\begin{equation}\label{omega1fromLemmacaseii}
 \omega_1=a_0\alpha_0+a_2\alpha_2+a_3\dx\theta.
\end{equation}
It follows by Lemma \ref{lemadtetaomega3} that only
\begin{equation}\label{certainhypo1}
 \omega_2=b_0\alpha_0+b_1\alpha_1+b_2\alpha_2, \qquad \omega_3=c_0\alpha_0+c_1\alpha_1+c_2\alpha_2
\end{equation}
may participate in a hypo structure $(\theta,\omega_1,\omega_2,\omega_2)$. The coefficients of these hypo structures of \textit{type II} must further solve the structural equations
\begin{equation}\label{certainhypo2}
 \begin{cases}
  b_1^2-b_0b_2=c_1^2-c_0c_2=a_3^2-a_0a_2\neq0 \\
  b_0a_2+b_2a_0=c_0a_2+c_2a_0=0 \\
  b_0c_2+b_2c_0-2b_1c_1=0.
 \end{cases}
\end{equation}
It follows easily that case $K=0$ does not admit hypo solutions of type II.

Let us also search for nearly-hypo structures with $\omega_1$ closed, of the type of well-known case ii, i.e. of the previous type. Therefore, over the same base manifold. We have system \eqref{generalnearlyhypo} and in particular $\omega_3$ determined by $\omega_2$. We have $a_1=0$ and we know the curvature, $K=\frac{a_0}{a_2s^2}$, which merely solves automatically the last equation in the system.

Double-hypo structures are the next interesting case. They are given by an extra condition, $\dx(\theta\wedge\omega_2)=0$, which implies $b_3=0$. The two systems above are then reduced to $a_2,a_3,p\neq0$ and
 \begin{equation}\label{certaindoublehypo}
  \begin{cases}
   K=\frac{a_0}{a_2s^2} \\
   a_3^2-a_0a_2=a_3p \\
   b_1^2-b_0b_2=a_3p \\
   a_0b_2+a_2b_0=0 \\
   b_0^2-2s^2Kb_0b_2+s^4K^2b_2^2+4s^2Kb_1^2=36s^4a_3p^3 .
  \end{cases}
 \end{equation}
 We call these structures the natural $\SU(2)$-structures on $\cals$ of type II.
\begin{teo}[\bf Double-hypo of type II]\label{noncontactdoublehypo}
 The natural $\SU(2)$-structures with 1-form $\tilde{\theta}=-2p\theta$ and closed 2-form $\omega_1$ from case ii of Lemma \ref{lemadomega1iguala0} are double-hypo if and only if they are given by (\ref{omega1fromLemmacaseii},\ref{omega3ofcertainlynearlyhypo},\ref{certaindoublehypo}) and $a_0a_2,a_3p>0$. Moreover, in this case $M$ has positive sectional curvature
\begin{equation}
 K=9s^2p^2 .
\end{equation}
\end{teo}
\begin{proof}
 On the lhs of the last equation in the system, we have
 \begin{eqnarray*}
  (b_0-s^2Kb_2)^2+4s^2Kb_1^2 &=& (b_0-\frac{a_0b_2}{a_2})^2+4\frac{a_0}{a_2}(a_3p +b_0b_2)   \\
   &=& \frac{1}{a_2^2}\bigl((b_0a_2-a_0b_2)^2+4a_0a_2a_3p+4a_0a_2b_0b_2 \bigr) \\
   &=& \frac{4a_0a_3p}{a_2}.
 \end{eqnarray*}
The rhs yields the identity $\frac{a_0}{a_2}=9s^4p^2$ and the result follows. The condition $a_0a_2,a_3p>0$ is required by \eqref{themetriconLperp} and can only be proved later (Proposition \ref{themetric_omega1oftypeii}).
\end{proof}
Thus we are bound to positive sectional curvature. 

Notice that $\dx\tilde{\theta}\neq-2\omega_1$, so these double-hypo structures are not contact hypo. Yet we have the following result, which contrasts, for instance, with the double-hypo in \eqref{contacdoublehypoR3xS2nonSE}.
\begin{prop}
 All double-hypo structures of type II satisfy
 \begin{equation}
  \dx\omega_3=-3\tilde{\theta}\wedge\omega_2.
 \end{equation}
\end{prop}
\begin{proof}
  First we notice that $b_0+s^2Kb_2=(b_0a_2+b_2a_0)/a_2=0$. Then we wish to check $\dx\omega_3=6p\theta\wedge\omega_2$, this is,
  \begin{eqnarray*} 
 \frac{Kb_1}{3ps^2}\theta\wedge\alpha_1+\frac{s^2Kb_2-b_0}{6s^2p}(\frac{2}{s^2}\theta\wedge\alpha_2-2K\theta\wedge\alpha_0)+ \qquad\qquad\  & &\\ 
 + \frac{b_1}{3ps^2}K\theta\wedge\alpha_1 =6p\theta\wedge(b_0\alpha_0+b_1\alpha_1+b_2\alpha_2) . & &
 \end{eqnarray*}
 This is equivalent to the system
 \[ -2s^2K^2b_2+2Kb_0=36s^2p^2b_0,\quad \frac{2Kb_1}{3s^2p}=6pb_1, \quad 2s^2Kb_2-2b_0=36s^4p^2b_2  \]
 or
 \[\ \  -s^2Kb_2+b_0=2b_0,\qquad Kb_1=9s^2p^2b_1, \qquad 
         s^2Kb_2-b_0=2s^2Kb_2 .   \]
 Since these three equations are satisfied, the result follows.
\end{proof}
In this case it seems there is a real $\SU(2)$ rather than $\SO(2)$ irreducibility, for in this hypothesis the last result is quite easy to prove from the structure equations and letting $\dx\omega_3$ be a linear combination of the $\theta\wedge\omega_i$.

\vspace{2mm}

\noindent
\textsc{Examples}.
The following give two double-hypo structures, not contact-hypo.
\begin{itemize}
   \item 
   With $b_0=b_2=0,\ a_3p=1$ and radius $s=1$, we have  $\tilde{\theta}=-2\theta,\ b_1^2=1,\ K=9p^2$ and still an interval of solutions; one example is with $K=5$
\begin{equation}\label{exampledoublehyp1}
\begin{split}
 \omega_1=2\alpha_0+\frac{2}{5}\alpha_2+\frac{3\sqrt{5}}{5}\dx\theta ,     \hspace{7mm} \\  \omega_2=\alpha_1,\qquad\quad
 \omega_3=3\alpha_0-\frac{\sqrt{5}}{5}\alpha_2 .
 \end{split}
\end{equation}
   \item  
   With $a_0=2,\ a_2=1,\ a_3=2,\ p=1$ and $s^2K=2$. This implies $s=\sqrt[4]{2/9}$ and $K=3\sqrt{2}$. For such arbitrary choices, there remains an interval of solutions; one example is
\begin{equation}\label{exampledoublehypo2}
 \begin{split}
 \omega_1=2\alpha_0+\alpha_2+2\dx\theta ,    \hspace{31mm}\\
   \omega_2=-\sqrt{2}\alpha_0\pm\alpha_1+\frac{\sqrt{2}}{2}\alpha_2 ,\qquad\quad
 \omega_3=\pm\sqrt{2}\alpha_0+\alpha_1\mp\frac{\sqrt{2}}{2}\alpha_2 .
 \end{split}
\end{equation}
\end{itemize}
\vspace{2mm}

Theorem \ref{noncontactdoublehypo} generalizes the $\SU(2)$-structure results found in \cite[Proposition 6.3]{FIMU}, which are computed directly on $S^3\times S^2$. Our family of double-hypo structures on $S^3\times S^2$ is one dimension higher. We remark that in \cite{FIMU} an auxiliary global parallel frame field on $S^3$ is used in order to deal with the differential geometry of the unit tangent sphere bundle of the 3-sphere.

\subsection{The metric explicit}
\label{Tme}

We provide some further information on the $\SU(2)$ metrics on $\cals_{s,M}$ with the most generic $\omega_i$, linear combination of $\alpha_0,\alpha_1,\alpha_2,\dx\theta$. For a given a set of coefficients $a_0,\ldots,c_3$, recall the metric induced on $\cals$ is denoted by $g_{_{\SU(2)}}$.

The tautological horizontal or Reeb vector field $-\frac{1}{2ps}\,e_0$ on $\cals$ is dual to $\tilde{\theta}=-2ps\,e^0$. We then must have $\|e_0\|_{\SU(2)}=2s|p|$ and $\ker\theta=\ker e^0=\ker\gsudois{e_0}{\ }$. Now we need to define the following functions:
\begin{equation}
 \begin{split}\label{simplythemetric}
  g_{11}=g_{22}&=(a_1b_0-a_0b_1)c_3+(a_0b_3-a_3b_0)c_1+(a_3b_1-a_1b_3)c_0 \\
  g_{33}=g_{44} &=(a_2b_1-a_1b_2)c_3+(a_1b_3-a_3b_1)c_2+(a_3b_2-a_2b_3)c_1 \\
  g_{12}=g_{34} &=0 \\
  g_{13}=g_{24}& =\tfrac{1}{2}(a_3(b_2c_0-b_0c_2)+b_3(a_0c_2-a_2c_0)+
           c_3(a_2b_0-a_0b_2)) \\
  g_{14}=-g_{23}& =\tfrac{1}{2}(a_1(b_0c_2-b_2c_0)+b_1(a_2c_0-a_0c_2)+
           c_1(a_0b_2-a_2b_0)) .
 \end{split}
\end{equation}
\begin{prop}\label{Proposition_themetricmatrix}
Let $e_0,e_1,\ldots,e_4$ be an adapted frame on $\cals$, hence orthonormal for the canonical metric. Then the symmetric matrix $G:=[\gsudois{e_i}{e_j}]_{1\leq i,j\leq4}$ equals $[g_{ij}]_{1\leq i,j\leq4}$, this is
\begin{equation} \label{themetricmatrix}
G= \left[\begin{array}{cccc}
        g_{11} & 0      & g_{13} & -g_{23}  \\
             0 & g_{11} & g_{23} & g_{13}  \\
        g_{13} & g_{23} & g_{33} &  0  \\
       -g_{23} & g_{13} &  0     & g_{33}  
       \end{array}  \right].
\end{equation}
\end{prop}
\begin{proof}
 A direct application of Theorem \ref{Teo_themetriconLperp}.
\end{proof}
We note that $G$ is indeed invariant of the choice of adapted frame, because that is the case of the fundamental exterior differential system. Or, more plainly, because $\C$ is abelian. Further on, of course we must have condition \eqref{hypostruceq2}, which is equivalent to the metric being positive definite due to Theorem \ref{Teo_themetriconLperp}. A computation first gives 
\begin{equation}\label{determinantedeG}
 \det G=(g_{11}g_{33}-g_{13}^2-g_{23}^2)^2 .
\end{equation}
Computing the minors of $G$ yields the following result.
\begin{prop}\label{Prop_Gmetricpositivedefinite}
 A natural $\SU(2)$ metric on $\cals$ being positive definite is equivalent to
 \begin{equation}
    g_{11}>0,\quad g_{11}g_{33}-g_{13}^2-g_{23}^2>0.
 \end{equation}
\end{prop}

The metric matrix $G$ announces a new class of natural metric on tangent sphere bundles of 3-manifolds, which to the best of our knowledge was never considered before. The structure yields the `g-natural' metrics known in the literature, as well as that new class. Recall the `g-natural' metrics, e.g. from \cite{Abb1,AbbKowal,DrutaRomaniucOproiu}, refer to a metric like the above but \textit{only} involving a constant linear combination of $\theta\otimes\theta$ and $g(x^h,y^h),g(x^h,y^v)$, $g(x^v,y^v)$ for the lifts of any $x,y\in TM$. Hence the importance by the negative of the next result.
\begin{prop}\label{SU2metricsareofmoregeneralkind}
A natural $\SU(2)$ metric on $\cals$ is a g-natural metric if and only if $g_{23}=0$.
\end{prop}
We remark there do exist structures with $g_{13}=0$ and $g_{23}\neq0$, cf. Proposition \ref{themetric_omega1oftypeii}.

Next we give a formula for the unique endomorphisms $\Phi_i\in\End{T\cals}$, for $i=1,2,3$, ortho\-go\-nal for the $\SU(2)$ metric and such that 
\begin{equation}
{\Phi_i}^*\omega_i=\omega_i,\qquad \qquad \Phi_i^2=-1_{T\cals}+e_0\otimes e^0.
\end{equation}
Taking any adapted frame and denoting the matrices of $\Phi_i,\omega_i$ restricted to $\ker\theta$ by the same letters, we have
\begin{equation}
 \omega_1=\left[\begin{array}{cc}
                 a_0J_1 & A_{13} \\ -A_{13}^T &  a_2J_1
                \end{array}\right]
\end{equation}
where ($k\in\N$)
\begin{equation}\label{definitionofJandA}
 J_k=\left[\begin{array}{cc}
                 0 & 1_k \\ -1_k & 0
    \end{array}\right] \qquad\mbox{and}\qquad  A_{13}=\left[\begin{array}{cc}
                 -a_3 & a_{1} \\ -a_1 & -a_3
                \end{array}\right] .
\end{equation}
Equivalent notations follow for $\omega_2,\omega_3$, with $B_{13},C_{13}$, respectively, in place of $A_{13}$. Recall there exists a unique $\nu\in\R,\ \forall i$,  such that $a_1^2+a_3^2-a_0a_2=\nu$, etc. So we have $A_{13}A_{13}^T=(a_0a_2+\nu)1_2$. Since $A_{13}J_1=J_1A_{13}$, we have $\omega_1\hat{\omega}_1=\nu1_4$ with
\begin{equation}
 \hat{\omega}_1=\left[\begin{array}{cc}
                 a_2J_1 & -A_{13} \\ A_{13}^T &  a_0J_1
                \end{array}\right].
\end{equation}
The $\SU(2)$-structure translates into $\omega_i\Phi_i=G$ for all $i=1,2,3$. This proves the formulae
\begin{equation}\label{Phi_imatrices}
 \Phi_i=\frac{1}{\nu}\hat{\omega}_iG .
\end{equation}

Next we deduce when an endomorphism, say $\Phi_1$, does preserve the vertical tangent bundle $V_0$, in which case we say simply $\Phi_1$ \textit{preserves the fibres} or \textit{preserves} $V_0$.
\begin{prop}
 $\Phi_1$ preserves the fibres if and only if
 \begin{equation}\label{conditionforpreservingthefibres}
 \begin{cases}
       a_2g_{23}+a_3g_{33}=0 \\  a_2g_{13}-a_1g_{33}=0
 \end{cases} .
 \end{equation}
 In particular, if the $\SU(2)$ metric is compatible with the canonical metric and $\Phi_1$ preserves the fibres, then $a_1=a_3=0,\ a_0=-a_2,\ b_0=-b_2,\ c_0=-c_2$.
\end{prop}
\begin{proof}
 Combining \eqref{Phi_imatrices} with \eqref{themetricmatrix}, condition $\Phi_1(V_0)\subset V_0$ is equivalent to the vanishing of the top right corner of $\Phi_1$. This is
 \[ a_2J_1\left[\begin{array}{cc}
         g_{13} & -g_{23}  \\ g_{23} & g_{13} 
       \end{array}  \right]-A_{13}g_{33}=0 ,    \]
and hence the system.

If, furthermore, we have $G=1_4$, then clearly $a_1=a_3=0$. And from $g_{11}=g_{33}=1$, we get $-a_0b_1c_3+a_0b_3c_1=1,\ a_2b_1c_3-a_2b_3c_1=1$ which yields $a_0=-a_2\neq0$ and the determinant $b_1c_3-b_3c_1\neq0$. Now from the formulae for $g_{13},g_{14}$, we find $b_3(a_0c_2-a_2c_0)+c_3(a_2b_0-a_0b_2)=0$ and $ b_1(a_2c_0-a_0c_2)+c_1(a_0b_2-a_2b_0)=0$. In other words,  $b_3(c_2+c_0)-c_3(b_0+b_2)=0$ and $b_1(c_0+c_2)-c_1(b_2+b_0)=0$. 
\end{proof}

Recall $\ker\theta=H_0\oplus V_0$, so it is only fair to consider the same question for horizontals: we say $\Phi_1$ \textit{preserves} $H_0$ if $\Phi_1(H_0)\subset H_0$. Equivalently,
\begin{equation}\label{conditionforpreservingHorizontals}
 \begin{cases}
       a_0g_{23}+a_3g_{11}=0 \\  a_0g_{13}-a_1g_{11}=0
 \end{cases} .
 \end{equation}

A last remark applies only to diagonal metrics, i.e. $g_{13}=g_{23}=0$. We recall the studies in \cite{Alb5,Alb2012} and specially \cite{Alb2014a} regarding a conformal change on the base metric on $M$, a radius $s$ of $\cals_{s,M}$, a conformal change on $H$ and $V_0$ and, moreover, how the previous three must relate, in order to build a homothety with the obvious map between tangent sphere bundles with different radius. Certainly noteworthy results in respect to classifying some of the $\SU(2)$ metrics above.

\section{The two distinguished types and evolution equations}

\subsection{The type I metrics}
\label{Atcomega_1igualadteta}

We resume with the natural $\SU(2)$-structures of type I, determined in Proposition \ref{ProptheinvariantSU2structures}, Theorems \ref{Teo_hypo} and \ref{Teo_nearlyhypo}. We have $a_0=a_1=a_2=b_3=c_3=0,\ a_3=1,\ b_1^2-b_0b_2=c_1^2-c_0c_2=1$, $b_0c_2+b_2c_0-2b_1c_1=0$, and therefore, reading from Proposition \eqref{Proposition_themetricmatrix}, we prove the next result.
\begin{prop}
 The natural $\SU(2)$ metrics of type I satisfy
\begin{equation}\label{themetricforomega1oftypei}
 \begin{split}
  g_{11}=g_{22}&=b_1c_0-b_0c_1 \\
  g_{33}=g_{44} &=b_2c_1-b_1c_2 \\
  g_{12}=g_{34} &=0 \\
  g_{13}=g_{24}& =\tfrac{1}{2}(b_2c_0-b_0c_2) \\
  g_{14}=-g_{23}& =0.
 \end{split}
\end{equation}
\end{prop}
Recall ``$g_{00}$''$=4s^2$ completes the information on this metric. The nearly-hypo structures of type I satisfy $g_{13}\neq0$ in general. For the particular case of the structure in \eqref{contacdoublehypoR3xS2nonSE}, over a flat base and radius $s$ with square $1/6$, we see
\begin{equation}\label{contacdoublehypoR3xS2nonSEmetrictensor}
 g_{11}=1,\quad g_{33}=5,\quad g_{13}=2.
\end{equation}
Regarding the Sasaki-Einstein metrics found in \eqref{S_equationsSasakiEinstein}, we have
\begin{equation}\label{S_equationsSasakiEinstein2}
 g_{11}=3s^2,\quad g_{33}=\frac{1}{3s^2},\quad g_{13}=0.
\end{equation}
\begin{prop}\label{Proposition_typeImetricpositivedefinite}
 For structures of type I, we have $\det G=1$. Moreover, the metric defined by the matrix $G$ is positive definite if and only if $b_1c_0-b_0c_1>0$.
\end{prop}
\begin{proof}
 The detailed computation, requiring \eqref{beesecees} together with the above results applied on \eqref{determinantedeG}, can be obviated if we notice that $\dx\theta=\omega_1$ induces the same volume-form as the canonical metric. However, there is more; the computation yields $g_{11}g_{33}-g_{13}^2-g_{23}^2=1$. The second assertion then follows by Proposition \ref{Prop_Gmetricpositivedefinite}.
\end{proof}
The following is a restatement of Proposition \ref{Propositioncompatiblemetric}, finally with a proof.
\begin{prop}\label{Propositioncompatiblemetric_withproof}
 The $\SU(2)$ metric of type I coincides on $\ker\theta$ with the canonical metric if and only if $b_0=-b_2=-c_1,\ b_1=c_0=-c_2$ and $b_0^2+b_1^2=1$.
\end{prop}
\begin{proof}
Conditions \eqref{beesecees} and $G=1_4$ lead to the equivalent relations.
\end{proof}

Immediately we see that $\Phi_1$ arising from $\omega_1=\dx\theta$ does not preserve the fibres. As matrices, we have $\omega_1=-J_2$, thus
\begin{equation}
  \Phi_1= J_2G=\left[\begin{array}{cc}
                 g_{13}1_2 & g_{33}1_2 \\ -g_{11}1_2 & -g_{13}1_2
    \end{array}\right] .
\end{equation}
In particular we verify that ${\Phi_1}^2=-1_4$. 

Now let us see $\Phi_2$ for general $\SU(2)$-structures of type I.
\begin{prop}
 $\Phi_2$ preserves the fibres if and only if $c_2=0$. 
 If moreover the metric is compatible with the canonical metric, then the case is that of the \emph{main example}.
\end{prop}
\begin{proof}
 By \eqref{conditionforpreservingthefibres} the condition is equivalent to $b_3g_{33}=0$ and $b_2g_{13}-b_1g_{33}=0$. Recalling \eqref{beesecees}, we have $b_3=0$. On the other hand,
  \begin{eqnarray*}
   b_2g_{13}-b_1g_{33} &=& \frac{1}{2} b_2^2c_0-\frac{1}{2}b_0b_2c_2-b_1b_2c_1+b_1^2c_2 \\  &=& \frac{1}{2} b_2^2c_0-\frac{1}{2}b_0b_2c_2-\frac{1}{2}b_0b_2c_2
 -\frac{1}{2}b_2^2c_0+c_2+b_0b_2c_2  \\
   &=& c_2.
  \end{eqnarray*}
  The result now follows easily.
\end{proof}
 One may verify as above that 
\begin{equation}
   \Phi_2\ \mbox{preserves}\ H_0\ \ \Longleftrightarrow\ \ c_0=0.
\end{equation}
\begin{coro}
 $\Phi_2$ preserves $H_0$ and $V_0$ if and only if $\pm b_2>0$ and
 \begin{equation}
  \omega_1=\dx\theta,\quad\omega_2=b_2\alpha_2-\frac{1}{b_2}\alpha_0,\quad\omega_3=\pm\alpha_1 .
 \end{equation}
\end{coro}

\subsection{The double-hypo structures of type II}
\label{TdhtIIm}

We return to the natural non contact double-hypo structures $(\tilde{\theta},\omega_1,\omega_2,\omega_3)$ of type II, found in Theorem \ref{noncontactdoublehypo}, in order to study the induced metric. However, our conclusion will be that this class of metrics on $\cals$ deserves a dedicated study.

\begin{prop}\label{themetric_omega1oftypeii}
Double-hypo structures of type II satisfy
 \begin{equation}
  g_{11}=\frac{a_0a_3^2}{3a_2s^2},\qquad 
  g_{33}=\frac{a_3^2}{3s^2},  \qquad    g_{13}=0,\qquad   
  g_{23}=-\frac{a_0a_3}{3s^2}.
 \end{equation}
 The positive definite condition on the metric corresponds to $a_0a_2>0,a_3p>0$.
\end{prop}
\begin{proof}
 Besides $a_1=b_3=c_3=0$ and $a_2,p\neq0$, we have system \eqref{certaindoublehypo} and
\[  c_0=\frac{Kb_1}{3p},\qquad c_1=-\frac{b_0}{3s^2p},\qquad 
 c_2=-\frac{b_1}{3s^2p} .  \]
Therefore $b_0+b_2Ks^2=(b_0a_2+b_2a_0)/{a_2}=0$ and hence $b_2c_0-b_0c_2=0$. On the other hand, by \eqref{simplythemetric}, we find immediately $g_{11}=a_3(b_1c_0-b_0c_1)$, $g_{33}=a_3(b_2c_1-b_1c_2)$, $g_{13}=\frac{1}{2}a_3(b_2c_0-b_0c_2)$ and $g_{23}=\frac{1}{2}(b_1(a_0c_2-a_2c_0)+c_1(a_2b_0-a_0b_2))$ and then the desired identities are trivial to deduce. Regarding the positive definite condition required by Proposition \ref{Prop_Gmetricpositivedefinite} we definitely must have $a_0a_2>0$. Since 
\[  g_{11}g_{33}-g_{13}^2-g_{23}^2=\frac{a_0a_3^4}{9a_2s^4}-\frac{a_0^2a_3^2}{9s^4}
 =\frac{a_0a_3^2}{9a_2s^4}(a_3^2-a_0a_2)=\frac{a_0a_3^2}{9a_2s^4}a_3p , \]
the result follows. 
\end{proof}
Using \eqref{conditionforpreservingthefibres} and \eqref{conditionforpreservingHorizontals} the following is trivial to check.
\begin{prop}
 For natural double-hypo structures of type II, neither $\Phi_1,\Phi_2$ or $\Phi_3$ preserve the horizontal or the vertical distributions.
\end{prop}


\subsection{Evolution equations from hypo structures}
\label{Eefnhs}

Let us recall a question raised in \cite{ContiSalamon} regarding an $\SU(2)$ structure on a 5-dimensional manifold $N$ and the associated $\SU(3)$ metric defined on $N\times\R$, cf. \eqref{hypostruceq3}.

The fundamental article on the \textit{generalized Killing spinors in dimension 5}, which introduces hypo structures, establishes when a smooth 1-parameter family of hypo structures $(\tilde{\theta},\omega_1,\omega_2,\omega_3)_t$ on $N$, time $t$ dependent, induces an integrable $\SU(3)$ (Calabi-Yau) metric on the product manifold via ($\phi=\omega_2+\sqrt{-1}\omega_3$)
\begin{equation}\label{SU3structurefromSU2}
     F=\omega_1+\tilde{\theta}\wedge\dx t,\qquad 
 \Psi =\Psi_++\sqrt{-1}\Psi_- =\phi\wedge(\tilde{\theta}+\sqrt{-1}\dx t).
\end{equation}
If $\omega_1,\tilde{\theta}\wedge\omega_2,\tilde{\theta}\wedge\omega_3$ are closed, then the evolution equations
\begin{equation}\label{evolutionequationshypotoKahler}
\begin{cases}
 \partial_t\omega_1=-\dx\tilde{\theta}\\ \partial_t(\omega_2\wedge\tilde{\theta})=-\dx\omega_3\\ \partial_t(\omega_3\wedge\tilde{\theta})=\dx\omega_2 
\end{cases}
\end{equation}
are easily deduced as the integrability equations $\dx F=\dx\Psi=0$, cf. \cite[Proposition 4.1]{ContiSalamon}. Reciprocally, an integrable product structure arising from a family of $\SU(2)$-structures implies the hypo equations \eqref{eqhypo} for all $t$.

In the analytic category, by Cartan-K\"ahler theory, \cite[Theorem 4.4]{ContiSalamon} establishes the existence of solution to \eqref{evolutionequationshypotoKahler}. The question remains open within the smooth category, quite puzzling due to the existence of non-analytic hypersurfaces in Calabi-Yau manifolds.

An explicit solution is immediately provided for Sasaki-Einstein manifolds, on $N\times\R_+$; it is known as the \textit{conical} $\SU(3)$-structure:
\begin{equation}\label{SU3structureforSasakiEinstein}
  F=t^2\omega_1+t\tilde{\theta}\wedge\dx t,\qquad 
 \Psi =t^2\phi\wedge(t\tilde{\theta}+\sqrt{-1}\dx t).
\end{equation}

Finally, one may consider the evolution equations on natural $\SU(2)$-structures on the total space $\cals_{s,M}$ of tangent sphere bundles and try to solve them within the same natural category. It is a quite demanding problem, also because there are other developments of the theory, namely in \cite{FIMU}, which involve the nearly-hypo structures and their own evolution equations now lifted to nearly-K\"ahler complex 3-folds. Interesting findings on double-hypo $\SU(2)$-structures and half-flat $\SU(3)$-structures lead to constructions of manifolds with $\Gdois$-holonomy. They all lead to further substantial questions applying on our context, so we leave the subject for the moment and point the reader to a future work.

Nevertheless, we shall give a new solution to the evolution equations of Conti-Salamon for one case on $\cals$ with a natural hypo structure of type I. Given a hypo structure of type $I$ by the usual constant values $p$ in $\tilde{\theta}=-2p\theta$ and $a_3,b_0,b_1,b_2,c_0,c_1,c_2$ in $\omega_1,\omega_2,\omega_3$ over a constant sectional curvature $K$ oriented 3-manifold, we wish to solve the evolution equations within the type I natural hypo structures. In other words, we wish to find $P,A_3,B_0,B_1,B_2,C_0,C_1,C_2$ functions of $t$, such that 
\begin{equation}\label{functionswewishtofindforhypoIevolutionproblem}
\begin{split}
&  \tilde{\theta}=-2P\theta,\quad\omega_1=A_3\dx\theta,\quad \omega_2=B_0\alpha_0+B_1\alpha_1+B_2\alpha_2,\quad \omega_3=C_0\alpha_0+C_1\alpha_1+C_2\alpha_2, \\
& \hspace{20mm}  B_1^2-B_0B_2=C_1^2-C_0C_2=A_3^2,\ \ B_0C_2+B_2C_0-2B_1C_1=0, \\
& \hspace{37mm}  A_3>0,\qquad  B_1C_0-B_0C_1>0,
\end{split}
\end{equation}
is a 1-parameter family of $\SU(2)$-structures solving \eqref{evolutionequationshypotoKahler} and containing the initial structure. Recall from Theorem \ref{Teo_hypo} that all these structures are automatically hypo.
\begin{prop}
 The natural type I evolution equations are equivalent to 
 \begin{equation}\label{evolutionequationsnaturalstructure}
  \begin{cases}
  \partial_tA_3=2P \\
   \partial_t(PC_0)=KB_1 \\  \partial_t(PC_1)=\frac{s^2KB_2-B_0}{2s^2} \\
   \partial_t(PC_2)=-\frac{B_1}{s^2}
  \end{cases}
  \qquad  
  \begin{cases}  
     \partial_t(PB_0)=-KC_1 \\  \partial_t(PB_1)=-\frac{s^2KC_2-C_0}{2s^2} \\
   \partial_t(PB_2)=\frac{C_1}{s^2} .
  \end{cases}
 \end{equation}
\end{prop}
The proof is immediate applying the usual formulae. In particular, if $P=p$ is constant, then both $B_1$ and $C_1$ satisfy
\begin{equation}
 \partial_{tt}^2X-\frac{K}{p^2s^2}X=0 
\end{equation}
and so all $B_i,C_i$ are in general of the \textit{elliptic} kind. However, the  assumption proves not to be so fruitful, because then $A_3=2pt+a_4$ and since the solutions must satisfy $B_1^2-B_0B_2=C_1^2-C_0C_2=A_3^2=4p^2t^2+4pa_4t+a_4^2,\ \forall t$, cf. \eqref{beesecees}, we easily run into contradiction. Clearly an exception occurs with the flat case, $K=0$, a we shall see below. 

Another assumption to make would be $P=p_1t+p_2$ with $p_1,p_2$ constant. This leads to quadratic solutions, but only for $K>0$ although not necessarily the Sasaki-Einsten conical solution \eqref{SU3structureforSasakiEinstein}.

\subsection{An integrable special-Hermitian structure} 
\label{AiSU3s}

Following the above discussion, we now solve the evolution equations for an oriented Riemannian flat 3-manifold $M$ and a natural hypo structure of type I on $\cals_{s,M}$.

We keep considering any radius $s$ tangent sphere bundle. Indeed, the variable $s$ may enter into the solution as a function of $t$, over the fixed smooth manifold $\cals$. The same is true for the curvature $K$, as long as a conformal change on $M$ carries along conveniently with any changes in $s$. These relations are well established in \cite{Alb2014a}, in particular for space forms.

In the present setting, we have $K=0$ and the initial data of $(\tilde{\theta},\omega_1,\omega_2,\omega_3)$ of type I are the usual constants $p,a_3,b_0,\ldots,c_2$.

Then we have the following solution of system \eqref{evolutionequationsnaturalstructure} with $P=p>0$:
\begin{equation}
 \begin{split}
   &  A_3=2pt+a_4,\qquad B_0=b_0, \qquad C_0=c_0, \\
 &  B_1=\frac{c_0}{2ps^2}t+b_4,\qquad C_1=-\frac{b_0}{2ps^2}t+c_4, \\
 & B_2=-\frac{b_0}{4p^2s^4}t^2+\frac{c_4}{ps^2}t+b_5,\qquad 
 C_2=-\frac{c_0}{4p^2s^4}t^2-\frac{b_4}{ps^2}t+c_5,
 \end{split}
\end{equation}
with $a_4,b_4,c_4,b_5,c_5$ real constants.
 
The conditions required by $\SU(2)$-structures follow:
\begin{equation}
 \begin{split} \label{moresystems}
 & \hspace{30mm}   2pt+a_4>0 , \\
 & \qquad\qquad  b_0^2+c_0^2=16p^4s^4 , \qquad   b_4c_0-b_0c_4=4p^2s^2a_4 , \\
 & b_4^2-b_0b_5=c_4^2-c_0c_5=a_4^2 ,  \qquad   b_0c_5+b_5c_0-2b_4c_4=0.
 \end{split}
\end{equation}
These equations come from the second line of \eqref{functionswewishtofindforhypoIevolutionproblem}. For instance, we have
$C_1^2-C_0C_2=A_3^2$ if and only if $\frac{b_0^2}{4p^2s^4}t^2-\frac{b_0c_4}{ps^2}t+c_4^2+\frac{c_0^2}{4p^2s^4}t^2+\frac{c_0b_4}{ps^2}t-c_0c_5 = 4p^2t^2+4pta_4+a_4^2$,
and thus three of the five equations follow. Notice $B_1C_0-B_0C_1>0$ holds trivially. 

Also, notice the substitution $a_4=a_3,\ b_4=b_1,\ c_4=c_1,\ b_5=b_2,\ c_5=c_2$ solves the third line and yields the initial structure at time $t=0$.

Regarding the general solution of system \eqref{evolutionequationsnaturalstructure}, notice the $B$s and the $C$s determine each other and, in the end, they determine $A_3$ and so finally $A_3$ determines $P$. Hence the solution is not very far from the above.

Finally we consider the \textit{main example} over an oriented flat 3-manifold $M$. Letting $2ps=1$ and $a_4=b_4=c_4=b_5=c_5=0$ and, moreover, changing $t/s$ for $t$, then we may just as well let $p=\frac{1}{2},\ s=1$. We have the following solution of the natural evolution equations:
\begin{equation}
 \tilde{\theta}=-\theta,\quad \omega_1=t\dx\theta,\quad \omega_2=t^2\alpha_2-\alpha_0, \quad \omega_3=t\alpha_1.
\end{equation}
And so we obtain a \textit{new} integrable $\SU(3)$-structure on $Z=\cals_{1,M}\times\R_+$:
\begin{equation}
 F=t\dx\theta-\theta\wedge\dx t,\quad\qquad \phi=\omega_2+\sqrt{-1}\omega_3,
\end{equation}
\begin{equation}
\begin{split}
 & \hspace{34mm} \Psi=  \phi\wedge(-\theta+\sqrt{-1}\dx t) =\\
 &     = \theta\wedge\alpha_0-t^2\theta\wedge\alpha_2-t\alpha_1\wedge\dx t-
      \sqrt{-1}(t\theta\wedge\alpha_1-t^2\alpha_2\wedge\dx t+\alpha_0\wedge\dx t).
\end{split}
\end{equation}
Indeed, in no trivial way becomes $Z$ an open subset of $\C^3$. Nor for any flat trivializing-neighborhood of $M$. We also recall
\begin{equation}
 \dx\alpha_0=\theta\wedge\alpha_1,\quad \dx\alpha_1=2\theta\wedge\alpha_2,\quad \dx\alpha_2=0,
\end{equation}
in order to prove $\dx F=\dx\Psi=0$.
\begin{Rema} 
To give a direct proof of the fundamental differential system formulae \eqref{introd_derivadasdastres2formas}, deduced twice in general in \cite{Alb2011arxiv,Alb2015a}, now defined over the Euclidean space, one may use co\-or\-di\-na\-tes $(x^1,x^2,x^3,u^1,u^2,u^3)$ on $\R^3\times S^2$ with $\sum(u^i)^2=1$ and the notation $\dx^{ijk,}=\dx x^i\wedge\dx x^j\wedge\dx x^k$, $\dx^{ij,k}=\dx x^i\wedge\dx x^j\wedge \dx u^k$. Then
 \begin{equation}
\theta=\sum u^i\dx x^i, \quad \alpha_0=\cyclic_{123}u^1\dx^{23,}, \quad\alpha_1=\cyclic_{123} u^1(\dx^{2,3}-\dx^{3,2}), \quad \alpha_2=\cyclic_{123} u^1\dx^{,23}  .
 \end{equation}
\end{Rema}

\ 

\ 

\ 

\textsc{R. Albuquerque}

{\small\texttt{rpa@uevora.pt}}

\ 

Departamento de Matem\'atica da Universidade de \'Evora

Centro de Investiga\c c\~ao em Mate\-m\'a\-ti\-ca e Aplica\c c\~oes

Rua Rom\~ao Ramalho, 59, 671-7000 \'Evora, Portugal

\ \\
The research leading to these results has received funding from the People Programme (Marie Curie Actions) of the European Union's Seventh Framework Programme (FP7/2007-2013) under REA grant agreement n\textordmasculine~PIEF-GA-2012-332209.

\end{document}